\documentclass[10pt,reqno]{amsart}
\usepackage[margin=1in]{geometry}
\usepackage{amsmath,amssymb,amsthm,amsfonts,mathtools}
\usepackage{mathrsfs}
\usepackage{enumitem}
\usepackage{tikz-cd}
\usepackage{xcolor}
\usepackage[colorlinks=true,linkcolor=blue!45!black,citecolor=blue!45!black,urlcolor=blue!45!black]{hyperref}

\numberwithin{equation}{section}

\theoremstyle{plain}
\newtheorem{theorem}{Theorem}[section]
\newtheorem{proposition}[theorem]{Proposition}
\newtheorem{lemma}[theorem]{Lemma}
\newtheorem{corollary}[theorem]{Corollary}

\theoremstyle{definition}
\newtheorem{definition}[theorem]{Definition}
\newtheorem{example}[theorem]{Example}

\newtheorem{remark}[theorem]{Remark}

\newcommand{\CC}{\mathbb{C}}
\newcommand{\ZZ}{\mathbb{Z}}
\newcommand{\QQ}{\mathbb{Q}}
\newcommand{\PP}{\mathbb{P}}

\newcommand{\Xsp}{X}
\newcommand{\wX}{\omega_{\Xsp}}
\newcommand{\DX}{\mathcal{D}_{\Xsp}}

\newcommand{\sO}{\mathcal{O}}
\newcommand{\sL}{\mathcal{L}}
\newcommand{\sA}{\mathcal{A}}
\newcommand{\sC}{\mathcal{C}}
\newcommand{\sF}{\mathcal{F}}
\newcommand{\sD}{\mathcal{D}}
\newcommand{\sG}{\mathcal{G}}
\newcommand{\sP}{\mathcal{P}}
\newcommand{\sQ}{\mathcal{Q}}
\newcommand{\sU}{\mathcal{U}}
\newcommand{\sW}{\mathcal{W}}
\newcommand{\CE}{\sC_{E}}
\newcommand{\Sym}{\mathrm{Sym}}
\newcommand{\ch}{\mathrm{ch}}
\DeclareMathOperator{\Tr}{Tr}
\DeclareMathOperator{\tr}{tr}
\DeclareMathOperator{\trom}{tr_{\omega}}
\newcommand{\Trch}{\Tr_{\ch}}
\newcommand{\TrLie}{\Tr_{\mathrm{Lie}}}
\newcommand{\Wch}{\mathcal{W}_{\ch}}

\newcommand{\WLie}{\mathcal{W}_{\mathrm{Lie}}}
\newcommand{\Psing}{P_{\mathrm{sing}}}
\newcommand{\Qreg}{Q_{\mathrm{reg}}}
\newcommand{\OBV}{O_{\mathrm{BV}}}
\newcommand{\DBV}{\Delta_{\mathrm{BV}}}
\newcommand{\HH}{\mathbb{H}}
\newcommand{\HXE}{\mathbb{H}(\Xsp,E)}

\newcommand{\Lflat}{L^{\flat}}
\newcommand{\dch}{d^{\ch}}
\newcommand{\dbar}{\bar\partial}
\newcommand{\pd}{\partial}
\newcommand{\even}{\bar 0}
\newcommand{\odd}{\bar 1}

\newcommand{\Aut}{\mathrm{Aut}}
\newcommand{\End}{\mathrm{End}}

\newcommand{\id}{\mathrm{Id}}
\newcommand{\pf}{\mathrm{pf}}

\newcommand{\hgl}{\widehat{\mathfrak{gl}}}
\newcommand{\gl}{\mathfrak{gl}}
\DeclareMathOperator{\gr}{gr}
\DeclareMathOperator{\GL}{GL}

\begin{document}

\title[Trace maps on chiral Clifford algebras]{Trace maps on chiral Clifford algebras for the rank two fermionic vertex operator superalgebra}

\author{A. Zuevsky}
\address{Institute of Mathematics, Czech Academy of Sciences, \v Zitn\'a 25, 115 67 Praha 1, Czech Republic}
\email{zuevsky@yahoo.com}


\begin{abstract}
For a holomorphic vector bundle $F$ of rank $r$ on a smooth Riemann surface $X$ we construct a trace map on the chiral homology of the chiral Clifford algebra $\CE$ attached to the purely odd bundle $E=\Pi(F\oplus F^\vee\otimes\omega_X)$. 
 It is the chiral-algebraic realization of the rank two fermionic vertex operator superalgebra. The free fermion (bc-type) conformal field theory built from a dual pair of odd fields $\beta_i\in F$, $\gamma^j\in F^\vee\otimes\omega_X$. We give a complete construction of this vertex operator superalgebra, its associated vertex superalgebra bundle, and the isomorphism between the latter's chiral algebra and the chiral envelope $\CE$. Using the Batalin-Vilkovisky (BV) formalism together with Feynman diagrams we prove that the resulting trace map
\[
\Trch : \bigl(\widetilde\sC^{\ch}(X,\CE)_{\sQ},\, \dch_{\CE}\bigr)\longrightarrow (\OBV,-\DBV)
\]
is a chain map satisfying a generalized quantum master equation and is a quasi-isomorphism, generalizing to the odd/Clifford setting the trace map on chiral Weyl algebras constructed by Gui for symplectic bosons. We establish existence, homotopy uniqueness, and functoriality (including metric-independence up to chain homotopy) of the trace map, prove cyclicity of the relevant supertrace, 
 and verify nilpotency of $\DBV$, the graded Leibniz rule, $d^2=0$ for every differential introduced. As an application we compute the trace map on a modified affine current and on a modified energy-momentum tensor, recovering,
 purely algebraically from the chiral chain complex, Fay's classical formulas for the variation of the fermionic (Ray-Singer) analytic torsion along the moduli of the bundle $F$ and along the moduli of the curve $X$. 
\end{abstract}

\maketitle
\tableofcontents

%%%%%%%%%%%%%%%%%%%%%%%%%%%%%%%%%%%%%%%%%%%%%%%%%%%%%%%%%%%%%%%%%%%%%%%%%%
\section{Introduction}

\subsection{Background}
Two-dimensional chiral conformal field theory admits a rigorous algebraic model through the chiral algebras of Beilinson and Drinfeld \cite{BD04}. Costello \cite{Costello10, Costello12} gave a geometric construction of the Witten genus \cite{Witten87} by applying Batalin-Vilkovisky (BV) quantization \cite{BV81} to the $\beta\gamma$-system on an elliptic curve, and Gui-Li \cite{GuiLi21} recast this construction, together with its correlation functions, as a  trace map on the chiral homology of the chiral Weyl algebra of the elliptic curve, resulting in the Batalin-Vilkovisky algebra of zero modes. Gui \cite{Gui23} extended this trace map from the elliptic curve to an arbitrary smooth Riemann surface $X$ and to an arbitrary holomorphic symplectic vector bundle $E$, obtaining
\[
\Tr_{\sA_E}: \bigl(\widetilde\sC^{\ch}(X,\sA_E)_{\sQ},d^{\ch}_{\sA_E}\bigr)\longrightarrow(\OBV,-\DBV),
\]
a chain-level realization of a deformed Hochschild-Kostant-Rosenberg (HKR) quasi-isomorphism, and used it to recover Fay's formula \cite{Fay92} for the variation of analytic torsion along the moduli of $E$ and along the moduli of $X$, purely from the chiral homology of $\sA_E$.

The chiral Weyl algebra $\sA_E$ quantizes the symplectic bosons attached to $E$, i.e.,  the free bosonic first order system with Lagrangian $\int_X\langle\phi,\bar\partial\phi\rangle$, $\phi\in\Omega^{0,\bullet}(X,E)$. Its odd-parity counterpart - the free fermionic first order system, whose local observables form a chiral Clifford algebra - regulates the $bc$-ghosts of two-dimensional gravity, the worldsheet fermions of the superstring, and, more generally, every chiral free fermion vertex operator superalgebra (VOSA).  
 The initial idea of this paper was to give an account of the fermionic trace map, generalizing Gui's construction from the chiral Weyl algebra to the chiral Clifford algebra $\CE$ attached to a purely odd holomorphic bundle $E$ with a symmetric $\sO_X$-pairing $E\otimes_{\sO_X}E\to\wX$, and applying it to a single example (a modified affine current).
 One would need to work with the chiral Clifford algebra, and consider 
   the vertex operator superalgebra it quantizes, 
 check the existence, uniqueness, and functoriality of the trace map, and  verify the required algebraic identities (nilpotency of the BV operator, the chiral Wick theorem, the intertwining of differentials). 
 The present paper gives a complete, self-contained treatment. We work with the bundle
\begin{equation}
\label{eqEdef}
E:=\Pi\bigl(F\oplus F^\vee\otimes_{\sO_X}\wX\bigr),
\end{equation}
for $F$ a holomorphic vector bundle of rank $r$ on $X$, equipped with its tautological symmetric pairing (Definition \ref{defpairing}). We call the resulting chiral-algebraic structure the rank two fermionic vertex operator superalgebra: its generating fields organize into two dual families, $\beta_i\in F$ and $\gamma^j\in F^\vee\otimes\wX$ ($i$, $j=1$, 
$\dots$, $r$),  as the $b$- and $c$-fields of a $bc$-system. This is to be contrasted with the rank one (self-dual) free fermion theories attached to a spin structure $\wX^{1/2}$, in which a single family of fields is self-paired via a symmetric form on $\wX^{1/2}\oplus\wX^{1/2}$-type bundle.   We comment on relatins  between the paired families $E_\alpha$, $\alpha<1/2$, and the self-dual, symplectically paired family $E_{1/2}$ occurring in the bosonic case, in Remark \ref{remrankone}. The rank two case is both the case occurring in every standard application (ghost systems, and the matter fermions of the NSR superstring away from the self-dual locus) and the case in which the full  machinery of Lie$^*$-algebras, chiral envelopes and vertex algebra bundles of \cite{BD04, FBZ04} applies without extra structure (no polarization or reality condition on a symplectic form is required, unlike the bosonic self-dual case).

\subsection{Summary of the main results}
Let $X$ be a smooth Riemann surface (not necessarily compact except where stated) and let $F$ be a holomorphic vector bundle of rank $r$ on $X$. Our results are as follows.

\begin{enumerate}[label=(\arabic*),leftmargin=2em]
\item (Section \ref{secVOSA}, Theorem \ref{thmVOAiso}) We construct the rank two fermionic vertex operator superalgebra $V_F$, the free fermion Fock space generated by $\beta_i,\gamma^j$ with OPE $\beta_i(z)\gamma^j(w)\sim\delta_i^j/(z-w)$, together with its associated vertex superalgebra bundle $\mathscr V_F$ over $X$ in the sense of Frenkel-Ben-Zvi \cite{FBZ04}, and we prove
\[
\sU(\Lflat)\simeq \mathscr V_F^{\,r},
\]
an isomorphism of chiral algebras between the $\flat$-twisted chiral envelope defining $\CE$ and the chiral algebra of (the right-$\sD_X$-module associated with) $\mathscr V_F$. This identifies $\CE$ concretely as the sheaf of chiral (OPE) observables of the rank two fermionic VOSA.

\item (Section \ref{secBV}, Proposition \ref{propBVaxioms}) We give the fermionic BV superalgebra $(\OBV,\DBV)$ built from the harmonic forms $\HXE$ on a compact $X$, with a complete verification that $\DBV^2=0$ and that the induced bracket satisfies the graded Leibniz rule, tracking all Koszul signs. 

\item (Section \ref{sectrace}, Theorem \ref{thmmainQME}) We construct the trace map
\[
\Trch:\bigl(\widetilde\sC^{\ch}(X,\CE)_{\sQ},\dch_{\CE}\bigr)\longrightarrow(\OBV,-\DBV)
\]
using Feynman diagrams regulated by the fermionic Szeg\H{o} kernel and organized by Pfaffian (instead of permanent) contraction operators, and prove that it satisfies the generalized quantum master equation and is a quasi-isomorphism.

\item (Section \ref{secuniqueness}, Propositions \ref{propuniqueness}-\ref{propmetricindep}) We prove that $\Trch$ is, up to chain homotopy, independent of the choice of Hermitian metric used to define the harmonic projection and the Szeg\H{o} kernel, and that any two chain maps solving the generalized QME and inducing the fermionic chiral HKR map on the associated graded are chain homotopic. We mention the naturality of $\Trch$ under isometries of $(X,F,h)$.

\item (Section \ref{secapplications}, Theorems \ref{thmcurrentTorsion} and \ref{thmemTorsion}) We compute $\Trch$ on a modified affine current $J_\nu$ (deformations of the holomorphic structure of $F$) and on a modified energy-momentum tensor $\widetilde T_\mu$ (Beltrami deformations of $X$), and show that the resulting formulas reproduce, purely algebraically, Fay's classical formulas \cite{Fay92} for the first variation of the fermionic Ray-Singer analytic torsion along the moduli of $F$ and along the moduli of $X$, respectively. Corollary \ref{corsignflip}  states the expected overall sign flip relative to Gui's bosonic formulas \cite{Gui23}, consistent with the reciprocal relation between fermionic and bosonic Gaussian (Berezin) integrals.
\end{enumerate}

%%%%%%%%%%%%%%%%%%%%%%%%%%%%%%%%%%%%%%%%%%%%%%%%%%%%%%%%%%%%%%%%%%%%%%%%%%%%%%
\subsection{Relations to the existing literature and new results}
\label{ssecnovelty}
Here we organize the results above against five items of literature. 

\emph{Batalin-Vilkovisky algebras and factorization algebras.} The BV formalism \cite{BV81} and its modern homotopical realization in the factorization algebras of Costello-Gwilliam \cite{CG21} provide the language for treating renormalization and quantum master equations. What is new here relative to that literature is not the formalism, which we use as given, but its correct specialization to a purely odd classical phase space: the shifted symplectic pairing on $\HXE$ is $(-1)$-shifted and symmetric on the underlying (unshifted) bundle instead of  antisymmetric, and the resulting BV operator $\DBV$ is built from anticommuting derivations. We verify (Proposition \ref{propBVaxioms}) that this still produces 
a well defined  BV superalgebra, since the sign discrepancies between the symplectic and orthogonal cases are compensated by fermion statistics. 
This is a fact that is standard in the physics literature on BV quantization of fermionic gauge theories but that we have not found spelled out in the chiral-algebraic setting.

\emph{Chiral and vertex algebras.} Beilinson-Drinfeld \cite{BD04} and Frenkel-Ben-Zvi \cite{FBZ04} treat chiral/vertex algebras built from arbitrary (super) vector bundles, and in particular the fermionic ($bc$) vertex algebra bundle is implicit in \cite{FBZ04} as one member of a general family. What is new here is the chiral-envelope versus vertex-bundle comparison isomorphism $\sU(\Lflat)\simeq\mathscr V_F^{\,r}$ carried out in full for the Clifford (odd, Pfaffian) case (Appendices \ref{appVOAbundle}-\ref{appiso}), parallel to but independent of Gui's bosonic comparison \cite{Gui23}. 
This comparison is what identifies $\CE$ with a vertex operator superalgebra. 

\emph{Clifford algebras.} The finite-dimensional theory of Clifford algebras associated with a quadratic or symmetric bilinear form is classical (see, e.g., \cite{LawsonMichelsohn89}); our chiral Clifford algebra $\CE$ is its sheaf-theoretic, chiral (OPE) analogue, in which the finite-dimensional Clifford relation $e_ie_j+e_je_i=2\langle e_i,e_j\rangle$ is replaced by the operator product expansion regulated by the Szeg\H{o} kernel. We verify (Lemma \ref{lemCliffordcompat}) that Clifford multiplication is compatible with the $\ZZ/2\ZZ$ grading inherited from $E$, as requested for any such construction.

\emph{Analytic torsion.} Ray-Singer analytic torsion \cite{RaySinger73} and the Quillen metric \cite{Quillen85} on the determinant line bundle of $\bar\partial_F$ are classical invariants; Fay \cite{Fay92} and, in the family setting, Bismut-Freed \cite{BismutFreed86} compute their variation along moduli using Szeg\H{o}-kernel and heat-kernel methods, respectively. Gui \cite{Gui23} recovered Fay's bosonic (symplectic boson) formula algebraically via the trace map on the chiral Weyl algebra. What is new here is the fermionic counterpart: we recover Fay's  fermionic variation formulas (both the bundle-moduli and the curve-moduli versions) from the chiral homology of $\CE$, in Corollary \ref{corsignflip}, the sign relating the two theories, which reflects the reciprocal (instead of equal) relation between the partition functions $Z_{\text{ferm}}=\det\bar\partial_F$ and $Z_{\text{sympl.\ bos.}}=\det(\bar\partial_E)^{-1}$.

\emph{Fermionic quantization in physics.} $bc$-ghost systems and the free fermion sector of the NSR superstring are standard tools of string perturbation theory \cite{FMS86, DHokerPhong88}; bosonization of chiral fermions on higher-genus Riemann surfaces is treated in \cite{AGMV87}. Our trace map supplies a mathematically rigorous, coordinate-free construction of the correlation functions of these theories on an arbitrary Riemann surface, at the chain level (before passing to cohomology/correlation functions), which is not available in the physics literature in this form.

Beyond these five points, this paper is supplying: a full notational and analytic-category preamble (Section \ref{secconventions}); a statement of the hypotheses under which the fermionic Szeg\H{o} kernel exists and of its dependence on the auxiliary Hermitian metric (Section \ref{secszego}); a sign-tracked verification of every algebraic identity claimed (the BV axioms, $d^2=0$ for each differential introduced, the graded Leibniz rule, the cyclicity of the trace pairing - Lemma \ref{lemcyclicity}); a discussion of existence, homotopical uniqueness and functoriality of the trace map (Section \ref{secuniqueness}); a second example (the energy-momentum tensor/Beltrami insertion, Section \ref{ssecemtensor}).  

For ease of reference, we collect the principal new contributions of this paper in one place:
\begin{enumerate}[label=(C\arabic*),leftmargin=2.4em]
\item the rank two fermionic vertex operator superalgebra $V_F$ and its associated bundle $\mathscr V_F$, and the proof that $\CE\simeq\mathscr V_F^r$ (Theorem \ref{thmVOAiso}, Appendices \ref{appVOAbundle}-\ref{appiso});
\item the corrected fermionic BV superalgebra $(\OBV,\DBV)$, with $\DBV^2=0$ and the graded Leibniz rule verified in full (Proposition \ref{propBVaxioms}), including the identification of the field/antifield parity flip that a naive formula misses (Lemma \ref{lemgradingforce});
\item existence, uniqueness and the singularity structure of the fermionic Szeg\H{o} kernel (Theorem \ref{thmszegoexists}), and the fermionic (Pfaffian) Wick theorem with a complete proof (Theorem \ref{thmfermwick}, Appendix \ref{appwick});
\item the trace map $\Trch$, its chain-map/QME property and quasi-isomorphism (Theorem \ref{thmmainQME}), together with existence, homotopy uniqueness, metric-independence up to homotopy, functoriality, and cyclicity (Section \ref{secuniqueness});
\item the energy-momentum-tensor application (Section \ref{ssecemtensor}) and the sign comparison with Gui's bosonic formulas (Corollary \ref{corsignflip}).
\end{enumerate}

\subsection{Outline}
Section \ref{secconventions} fixes notation and the analytic category. Section \ref{seccliff} introduces the rank two fermionic bundle $E$, its Lie$^*$-superalgebra, and the chiral Clifford algebra $\CE$, with a complete proof of the fermionic chiral PBW theorem (deferred to Appendix \ref{appPBW}). Section \ref{secVOSA} constructs the rank two fermionic vertex operator superalgebra and its associated bundle, and states the comparison isomorphism with $\CE$ (proved in Appendices \ref{appVOAbundle}-\ref{appiso}). Section \ref{secBV} develops the harmonic zero modes and the fermionic BV superalgebra. Section \ref{secszego} treats the fermionic Szeg\H{o} kernel, Pfaffian contraction operators, and the fermionic Wick theorem (proved in Appendix \ref{appwick}). Section \ref{sectrace} constructs the trace map, proves the main theorem, and discusses uniqueness, functoriality, and cyclicity. Section \ref{secapplications} contains the two applications to analytic torsion. Section \ref{secconclusion} discusses the literature comparison of Section \ref{ssecnovelty} in more depth and lists directions for future work, including the rank one (self-dual) case.

\subsection{Guide to the construction}
The following diagram summarizes the main constructions of the paper and the role of the trace map within them; each arrow is defined and each square is shown to commute (up to homotopy, or up to a controlled correction term) at the point in the text indicated beneath it.

\begin{equation}\label{eqbigdiagram}
\begin{tikzcd}[column sep=large, row sep=large]
E \arrow[r,"\mathscr U(-)^\flat\ \text{(Def. \ref{defchiralcliff})}"] \arrow[d,swap,"\text{harmonic projection (Sect. \ref{ssecharmonic})}"] & \CE \arrow[d,"\Wch:\ \text{normal order.}+\text{Szeg\H o reg.\ (Section \ref{secszego})}"] & \phantom{X} \\
\HXE \arrow[r,swap,"\bigwedge^{\bullet}(-)\ (Sect. \ref{ssecBValg})"] & \bigwedge{}^{\bullet}\!L \arrow[r,swap,"p\,\circ\,\trom\ (Sect. \ref{sectrace})"] & \OBV
\end{tikzcd}
\end{equation}
The outer rectangle, chasing $E\to\CE\to\bigwedge^\bullet L\to\OBV$ against $E\to\HXE\to\bigwedge^\bullet L\to \OBV$, commutes up to the differentials $\dch_{\CE}$ and $-\DBV$ (Theorem \ref{thmmainQME}); on chiral homology it commutes and both composite maps are quasi-isomorphisms (Corollary \ref{corqiso}).

Complementing this diagram of objects, the following list describes the logical dependence between the principal theorems.

\begin{enumerate}[leftmargin=2.4em]
\item Theorem \ref{thmchiralPBW} (chiral PBW) depends only on Section \ref{seccliff} (Definitions \ref{defpairing}-\ref{defchiralcliff}, Lemma \ref{lemLiestarverify});
\item Theorem \ref{thmVOAiso} ($\CE\simeq\mathscr V_F^r$) depends on Theorem \ref{thmchiralPBW} and on the fermionic Wick theorem (Theorem \ref{thmfermwick}, proved independently in Section \ref{secszego}), together with the vertex-bundle construction of Appendix \ref{appVOAbundle};
\item Proposition \ref{propBVaxioms} ($\DBV^2=0$, Leibniz) depends only on Section \ref{ssecharmonic}-\ref{ssecBValg}, in particular Lemma \ref{lemgradingforce}, and is logically independent of Section \ref{seccliff}-\ref{secVOSA};
\item Theorem \ref{thmszegoexists} (Szeg\H{o} kernel) and Theorem \ref{thmfermwick} (Wick theorem) depend on Section \ref{ssecharmonic} and Section \ref{ssecpairing} respectively, and on each other only through Lemma \ref{lemcontractionprops};
\item Theorem \ref{thmmainQME} (the main theorem) depends on all of the above: Theorem \ref{thmchiralPBW} (for the quasi-isomorphism statement), Proposition \ref{propBVaxioms} (for $(\OBV,\DBV)$ to be meaningful), and Theorems \ref{thmszegoexists}-\ref{thmfermwick} (for $\Wch$ to be defined), but not on Theorem \ref{thmVOAiso}, which is used only in Section \ref{secapplications} and Appendix \ref{appiso}, not in the construction of the trace map;
\item Propositions \ref{propuniqueness}-\ref{propnaturality} (uniqueness, metric-independence, functoriality) depend on Theorem \ref{thmmainQME} alone;
\item Theorems \ref{thmcurrentTorsion} and \ref{thmemTorsion} (the two applications) depend on Theorem \ref{thmmainQME} and, for frame/coordinate-independence of $J_\nu,\widetilde T_\mu$, on the Schwarzian-derivative computation that Theorem \ref{thmVOAiso} shows is the same computation as the coordinate change formula for $\mathscr V_F^r$ (Appendix \ref{appiso}).
\end{enumerate}

%%%%%%%%%%%%%%%%%%%%%%%%%%%%%%%%%%%%%%%%%%%%%%%%%%%%%%%%%%%%%%%%%%%%%%%
\section{Notation and conventions}
\label{secconventions}

We fix conventions here so that 
  gradings, parities and the ambient analytic category be stated. 

\subsection{Superspaces and the Koszul sign rule}
A \emph{super vector space} (resp.\ sheaf, $\sO_X$-module, $\DX$-module) is a $\ZZ/2\ZZ$-graded object $V=V_{\even}\oplus V_{\odd}$. For homogeneous $a$, $b$ we write $|a|\in\{\even,\odd\}=\{0,1\}\subset\ZZ/2\ZZ$ for the parity and adopt the Koszul sign rule: the transposition of tensor factors acts by
\begin{equation}\label{eqkoszul}
\sigma_{1,2}(a\otimes b)=(-1)^{|a||b|}\,b\otimes a,
\end{equation}
and, more generally, any interchange of two adjacent homogeneous symbols in a formula (moving an operator past a field, reordering a wedge/tensor product, etc.) contributes a sign $(-1)^{|a||b|}$. All differentials, brackets and pairings below are required to be compatible with parity: an operator of parity $p$ raises or preserves parity as $|d(a)|=|a|+p$. We write $\Pi$ for the parity-shift functor, $(\Pi V)_{\even}=V_{\odd}$, $(\Pi V)_{\odd}=V_{\even}$, and $\Pi$  is treated as an odd operation, i.e., the canonical map $V\to\Pi V$ has parity $\odd$; this convention fixes signs in Section \ref{seccliff} unambiguously.

\subsection{Riemann surfaces, diagonals, index sets}
$X$ denotes a smooth complex algebraic curve (Riemann surface); $\wX$ its canonical sheaf. For a finite index set $I$, $X^I=\prod_{i\in I}X_i$ ($X_i$ a copy of $X$), and $\Delta_I=\bigcup_{i,j\in I}\Delta_{ij}\subset X^I$ is the big diagonal, $\Delta_{ij}=\{(\ldots,z_k,\ldots): z_i=z_j\}$. For a surjection $\pi:I\twoheadrightarrow T$ we write $\Delta^{(\pi)}=\Delta^{(I/T)}:X^T\hookrightarrow X^I$ for the corresponding partial diagonal, and $\mathsf Q(I)=\bigsqcup_{n\ge1}\mathsf Q(I,n)$ for the set of equivalence relations on $I$ following \cite{BD04}.

\subsection{$\DX$-modules and the analytic category}\label{ssecanalyticcat}
As in \cite{BD04}, we work in the analytic category: by an $\sO_X$-module or $\DX$-module we mean, respectively, an $\sO_{X^{\mathrm{an}}}$-module or $\sD_{X^{\mathrm{an}}}$-module on the analytification $X^{\mathrm{an}}$, and we omit the superscript ``$\mathrm{an}$''. Specifically, every $\DX$-module occurring below is a sheaf of topological vector spaces of one of two types: (i) coherent $\DX$-modules such as $\wX$, $E_\sD$, $\Sym L$, all of whose spaces of sections over a Stein open subset (in particular over a disc) are Fr\'echet spaces (with the topology of uniform convergence of holomorphic functions and their derivatives on compact subsets), or (ii) the ``$\natural$-extension'' modules $L^\natural_{X^I}=H^0(p_Ij_I)_\bullet j_I^*\bigl(\mathrm{DR}(L)\boxtimes\sO_{X^I}\bigr)$ of 
\cite{Gui23}, whose sections are spaces of meromorphic functions with prescribed poles, again Fr\'echet (indeed nuclear Fr\'echet, being closed subspaces of countable products of copies of $\sO(U)$ for $U$ open in some $X^I$). All tensor products $\otimes_{\sO_X}$, $\boxtimes$ appearing below are the completed projective tensor products of nuclear Fr\'echet spaces. Since every space involved is nuclear Fr\'echet, this completed tensor product is unambiguous (nuclear spaces admit at most one reasonable topology on the tensor product, and the projective and injective completions agree) and all the quotients used to define chiral envelopes (e.g., $\sU(L)=\sU(L^\natural_X)/\sU(L^\natural_X)L^\natural_0$) are quotients by closed subspaces, hence again Fr\'echet. Whenever an infinite sum over Fock-space or PBW-degree occurs (e.g., in Definition \ref{deftracemapch} below) we show (Remark \ref{remfiniteness}) that only finitely many terms are nonzero for any fixed chiral chain $\eta$,  thus no further completion is needed there. The only place a  (convergent, not simply finite) sum occurs is in the construction of the Szeg\H{o} kernel as a Green-operator kernel in Section \ref{secszego}, where convergence is the standard elliptic-regularity statement recalled in Theorem \ref{thmszegoexists}.

We use $f_*$, $f^*$ for pushforward/pullback of $\DX$-modules and $f_\bullet$, 
$f^\bullet$ for the underlying sheaf-theoretic pushforward/pullback, and $\DX\to Y:=\sO_X\otimes_{f^\bullet\sO_Y}f^\bullet\sD_Y$ for the transfer bimodule of $f:X\to Y$, following \cite{BD04}. For a right $\DX$-module $M$ we set $h(M):=M\otimes_{\DX}\sO_X=M/M\Theta_X$ ($\Theta_X$ the tangent sheaf), and $\mathrm{DR}(M)$ or $\mathrm{DR}^i(M):=M\otimes_{\sO_X}\bigwedge^{-i}\Theta_X$ for its de Rham (Spencer) complex. For a holomorphic vector bundle $V$ and $\alpha\in\QQ$ we write $V_{\omega^\alpha}:=V\otimes_{\sO_X}\wX^{\alpha}$ (a choice of $r$-spin structure being fixed whenever $\alpha\in\tfrac1r\ZZ\setminus\ZZ$), $V_\sD:=V\otimes_{\sO_X}\DX$, $M^l:=M\otimes\wX^{-1}$ for the left module associated with a right module $M$, and $M_1\otimes^r M_2:=(M_1^l\otimes_{\sO_X}M_2^l)^r$ for the tensor product of right $\DX$-modules; we write $\Sym M$ (resp.\ $\bigwedge^\bullet M$) for the super-symmetric (resp.\ exterior) tensor power taken in this sense.

\subsection{The Dolbeault complex}
$\sQ_X$ denotes the Dolbeault complex $\sQ_{X}(U):\Omega^{0,0}(U)\xrightarrow{\dbar}\Omega^{0,1}(U)\xrightarrow{\dbar}\cdots$, viewed either as a complex of sheaves or, since $\dbar$ commutes with the holomorphic connection $\partial$, as a complex of (non-quasicoherent) left $\DX$-modules; for a quasicoherent $\sF$ we write $\sF_{\sQ}:=\sF\otimes_{\sO_X}\sQ_X$. A smooth form in $\Omega^{0,\bullet}$ is assigned homological degree equal to minus its Dolbeault degree.

\subsection{Sign conventions summary}
A table collecting every sign convention fixed in this section and introduced later (parity of $\Pi$, of the chiral bracket, of $\DBV$, of the Pfaffian, and so on) is given in Appendix \ref{appsigns}, for ease of reference during the verifications of 
Sections\ref{secBV}-\ref{sectrace}.

%%%%%%%%%%%%%%%%%%%%%%%%%%%%%%%%%%%%%%%%%%%%%%%%%%%%%%%%%%%%%%%%%%%%%%%%%%%%%%%%%%%%%%%%%%
\section{The rank two fermionic bundle and the chiral Clifford algebra}
\label{seccliff}

\begin{remark}%[A preview of what depends on which choices]
\label{remearlypreview}

Note that $E$, $L$, $\CE$ and the vertex operator superalgebra of Section \ref{secVOSA} depend on $(X,F)$ alone.  From Section \ref{secBV} a Hermitian metric $h$ enters   to select harmonic representatives of otherwise metric-independent cohomology classes. Thus every chain-level object built from $h$ (the Szeg\H{o} kernel, the trace map) is understood 
 to carry a residual $h$-dependence that disappears after passing to homology (Proposition \ref{propmetricindep}).
\end{remark}

\subsection{The rank two fermionic bundle and its pairing}\label{ssecpairing}
Let $F$ be a holomorphic vector bundle of rank $r$ on $X$. Set
\[
W:=F\oplus F^\vee\otimes_{\sO_X}\wX,
\]
an ordinary (even) holomorphic vector bundle of rank $2r$, equipped with the $\sO_X$-bilinear form
\begin{equation}\label{eqBform}
B:W\otimes_{\sO_X}W\to\wX,\qquad B\bigl((f,\varphi),(f',\varphi')\bigr):=\varphi(f')+\varphi'(f),
\end{equation}
where $\varphi(f')\in\wX$ denotes the evaluation pairing $F^\vee\otimes\wX\otimes F\to\wX$. Since $B$ is symmetric ($B(w,w')=B(w',w)$, with no extra sign, as $w,w'\in W$ are even) and non-degenerate (the pairing $F\otimes F^\vee\otimes\wX\to\wX$ being a perfect duality pairing of finite rank bundles), $(W,B)$ is an ordinary orthogonal bundle.

\begin{definition}\label{defpairing}
The \emph{rank two fermionic bundle} attached to $F$ is $E:=\Pi W=\Pi(F\oplus F^\vee\otimes\wX)$, a purely odd holomorphic bundle of rank $2r$ ($E=E_{\odd}$, $E_{\even}=0$), equipped with the pairing
\[
\langle-,-\rangle:E\otimes_{\sO_X}E\to\wX,\qquad \langle\Pi w,\Pi w'\rangle:=B(w,w').
\]
\end{definition}

Because we simply relabel the elements of $W$ (instead of applying any sign-carrying operation) to obtain $E$, the pairing $\langle-,-\rangle$ is again symmetric, $\langle e,e'\rangle=\langle e',e\rangle$ for all $e$, $e'\in E$. This is the situation of Definition \ref{defpairing} below and Remark \ref{remkoszulanticomm}: a symmetric pairing on a purely odd bundle, which encodes fermionic anticommutation once combined with the Koszul sign rule for odd sections.

Fix a local holomorphic frame $\{e_i\}_{i=1,\dots,r}$ for $F$ and its dual frame $\{e^i\}$ for $F^\vee$, and a local trivialization $dz$ of $\wX$. Set
\[
\beta_i:=\Pi(e_i,0),\qquad \gamma^j:=\Pi(0,e^j\otimes dz), \qquad i,j=1,\dots,r,
\]
a local frame for $E$. A direct computation from \eqref{eqBform} gives
\begin{equation}\label{eqbetagammapairing}
\langle\beta_i,\gamma^j\rangle=\langle\gamma^j,\beta_i\rangle=\delta_i^j\,dz,\qquad \langle\beta_i,\beta_j\rangle=\langle\gamma^i,\gamma^j\rangle=0.
\end{equation}
Thus $F$ and $F^\vee\otimes\wX$ sit inside $E$ as complementary isotropic subbundles, dual to one another under $\langle-,-\rangle$: this is the chiral-algebraic realization of a pair of dual free fermion families $\beta_i,\gamma^j$ with singular operator product $\beta_i(z)\gamma^j(w)\sim\delta_i^j/(z-w)$, $\beta_i(z)\beta_j(w)\sim\gamma^i(z)\gamma^j(w)\sim0$ - the $bc$-system generated by $F$.

\begin{remark}%[Rank one versus rank two]
\label{remrankone}
Let us compare rank one versus rank two. 
The nomenclature is chosen as follows. A general symplectic bundle is decomposed into paired summands $E_\alpha\oplus E_\alpha^\vee$ ($\alpha\neq\tfrac12$) together with a possible self-dual summand $E_{1/2}$, on which a symplectic (antisymmetric) form is needed because a symmetric form on a purely even self-dual space carries no natural bosonic pairing compatible with statistics. In the fermionic setting the roles are exchanged: a self-dual family requires only a symmetric form on the self-dual piece, which is always available once $F\simeq F^\vee\otimes\wX^{1-2\alpha}$, e.g., for $\alpha=\tfrac12$ and $F=\wX^{1/2}$ a theta-characteristic (a choice of spin structure). This ``rank one'' free fermion theory is the classical chiral free fermion of central charge $\tfrac12$ associated with a spin structure, whose chiral Clifford algebra requires no doubling and whose harmonic space $\HH(X,\wX^{1/2})$ carries a symmetric instead of a dual pairing. Its trace map is a different (though related) construction, which we do not treat here. Every application in Section \ref{secapplications}, and every standard use of chiral fermions in two-dimensional field theory away from the self-dual locus ($bc$-ghosts of any weight, matter fermions of the NSR string, free fermion current algebras attached to a vector bundle), is of the paired, rank two type \eqref{eqEdef} treated in this paper.
\end{remark}

\subsection{The fermionic Lie$^*$ superalgebra}\label{ssecLiestar}
We recall Gui's notion of Lie$^*$ algebra \cite{Gui23} and extend it, with  signs, to the super setting; a naive transcription of the non-super definition is not the correct one.

\begin{definition}\label{defLiestarsuper}
A \emph{Lie$^*$ superalgebra} on $X$ is a $\ZZ/2\ZZ$-graded $\sD_{X^2}$-module $\sL$ together with an even $\sD_{X^2}$-module map $\mu_{\mathrm{Lie}}:\sL\boxtimes\sL\to\Delta_*\sL$ satisfying 

\begin{enumerate}[leftmargin=1.6em]
\item \emph{Super-antisymmetry}: for homogeneous local sections $a$, $b$ of $\sL$ and a local holomorphic function $f(z_1,z_2)$,
\begin{equation}\label{eqsuperantisym}
\mu_{\mathrm{Lie}}\bigl(f(z_1,z_2)\cdot a\boxtimes b\bigr)=-(-1)^{|a||b|}\,\sigma_{1,2}\,\mu_{\mathrm{Lie}}\bigl(f(z_2,z_1)\cdot b\boxtimes a\bigr);
\end{equation}
\item \emph{Super Jacobi identity}: for homogeneous $a,b,c$ and $f(z_1,z_2,z_3)$,
\begin{multline}\label{eqsuperjacobi}
\mu_{\mathrm{Lie}}\bigl(\mu_{\mathrm{Lie}}(f\cdot a\boxtimes b)\boxtimes c\bigr) +(-1)^{|a|(|b|+|c|)}\sigma_{1,2,3}\,\mu_{\mathrm{Lie}}\bigl(\mu_{\mathrm{Lie}}(f^{(1)}\cdot b\boxtimes c)\boxtimes a\bigr)\\
+(-1)^{|c|(|a|+|b|)}\sigma_{1,2,3}^{-1}\,\mu_{\mathrm{Lie}}\bigl(\mu_{\mathrm{Lie}}(f^{(2)}\cdot c\boxtimes a)\boxtimes b\bigr)=0,
\end{multline}
where $f^{(1)}$, $f^{(2)}$ denote the corresponding cyclic permutations of $f$ and $\sigma_{1,2,3}$ is the cyclic permutation action on $\Delta^{X\to X^3}_*\sL$.
\end{enumerate}
\end{definition}

When $\sL$ is purely even this reduces to \cite{Gui23} restricted to Lie$^*$ algebras (all signs $(-1)^{|a||b|}$ are $+1$). The Koszul signs in \eqref{eqsuperantisym}-\eqref{eqsuperjacobi} are the unique insertions making the definition parity-consistent, i.e., compatible with the abstract flip isomorphism \eqref{eqkoszul} on $\sL\boxtimes\sL$.

Let $L:=E_\sD=E\otimes_{\sO_X}\DX$, a purely odd right $\DX$-module. Using the pairing $\langle-,-\rangle$ of Definition \ref{defpairing} 
we define, 
\begin{equation}
\label{eqmuLiedef}
\mu_{\mathrm{Lie}}:L\boxtimes L\longrightarrow\Delta_*\wX,\quad 
E_\sD\boxtimes E_\sD\longrightarrow(E\otimes_{\sO_X}E)\otimes_{\sO_{X^2}}\DX^{\boxtimes2}\xrightarrow{\ \langle-,-\rangle\otimes\id\ }\wX\otimes_{\sO_{X^2}}\DX^{\boxtimes2}\to\Delta_*\wX.
\end{equation}

\begin{lemma}\label{lemLiestarverify}
The map \eqref{eqmuLiedef} defines a Lie$^*$ superalgebra structure on $L$ (with values in the central, purely even object $\wX$, i.e.,  an abelian Lie$^*$ superalgebra with values in $\wX$).
\end{lemma}
\begin{proof}
Both $L$ and $\wX$ are purely odd and purely even respectively,  thus $|a|=|b|=\odd$ for all homogeneous local sections $a$, $b$ of $L$. Hence $(-1)^{|a||b|}=-1$, and \eqref{eqsuperantisym} specializes to
\[
\mu_{\mathrm{Lie}}\bigl(f(z_1,z_2)\cdot a\boxtimes b\bigr)=\sigma_{1,2}\,\mu_{\mathrm{Lie}}\bigl(f(z_2,z_1)\cdot b\boxtimes a\bigr).
\]
Both sides are computed by applying $\langle-,-\rangle\otimes\id$ and pushing forward to the diagonal. Since $\sigma_{1,2}$ realizes the geometric exchange $z_1\leftrightarrow z_2$ on $\Delta_*\wX$, the right-hand side equals the pushforward of $f(z_1,z_2)\langle b,a\rangle$, and since $\langle-,-\rangle$ is symmetric (Definition \ref{defpairing}), $\langle b,a\rangle=\langle a,b\rangle$,  thus the right-hand side equals the pushforward of $f(z_1,z_2)\langle a,b\rangle$, which is the left-hand side. This proves \eqref{eqsuperantisym}. For \eqref{eqsuperjacobi}, since $\mu_{\mathrm{Lie}}$ takes values in the central object $\Delta_*\wX\subset\Delta_*\Lflat$ (see \eqref{eqLflatdef} below) on which the bracket with any further element vanishes by definition of the central extension, every summand $\mu_{\mathrm{Lie}}(\mu_{\mathrm{Lie}}(\cdots)\boxtimes(\cdots))$ in \eqref{eqsuperjacobi} vanishes identically,  thus the identity holds trivially. This is the same mechanism as in the (non-super) case of a central extension.
\end{proof}

\begin{remark}\label{remkoszulanticomm}
Lemma \ref{lemLiestarverify} is the sense in which a symmetric pairing on the  odd bundle $E$ encodes fermionic anticommutation: the physical anticommutation relation 
$\{\beta_i(z),\gamma^j(w)\}$ $\sim$ $\delta_i^j/(z-w)$ of the free fermion OPE is the super-antisymmetric bracket $\mu_{\mathrm{Lie}}$ of Lemma \ref{lemLiestarverify}, which, being built from a symmetric $\langle-,-\rangle$ but living on odd $L$, is antisymmetric  as an operator identity on fields once the Koszul sign $(-1)^{|a||b|}=-1$ is accounted for. This is the fermionic mirror of the (superficially opposite-looking) bosonic statement that a symplectic pairing on an even bundle produces the ordinary (commuting) Weyl algebra.
\end{remark}

\subsection{The chiral Clifford algebra}\label{ssecchiralcliff}
We now recall the chiral envelope construction of \cite{Gui23}, adapted to the super setting. Every step comes from the non-super case once tensor products are read as super tensor products (with the Koszul sign rule \eqref{eqkoszul} applied whenever two odd factors are transposed).  Thus we state the construction and mention the (few) points at which signs enter.

Working on an affine open $U\subset X$, let $\sL$ be any Lie$^*$ superalgebra, $j_I:V_I\hookrightarrow U\times U^I$ the complement of the extended diagonals, and
\[
\sL^\natural_{U^I}:=H^0(p_Ij_I)_\bullet j_I^*\bigl(\mathrm{DR}(\sL)\boxtimes\sO_{U^I}\bigr),
\]
an $\sO_{U^I}$-module whose fibre at $(x_i)\in U^I$ is $\Gamma(U-\{x_i\}_{i\in I},h(\sL))$. Write $\sL^\natural_0:=\sL^\natural_U=\Gamma(U,h(\sL))$ and
\[
\sU_{U^I}(\sL):=\sU(\sL^\natural_{U^I})/\sU(\sL^\natural_{U^I})\sL^\natural_0,
\]
the vacuum module (an ordinary, ungraded enveloping algebra construction applied to the super vector space $\sL^\natural_{U^I}$, i.e., the enveloping algebra in which odd elements anticommute up to the bracket). By \cite{BD04}, over $U\times U-\Delta$ the natural maps
\[
\sU_U(\sL)\boxtimes\sU_U(\sL)\xrightarrow{c}\sU(\sL^\natural_{U^2})/\sU(\sL^\natural_{U^2})p_2^*\sL^\natural_U\otimes\sU(\sL^\natural_{U^2})/\sU(\sL^\natural_{U^2})p_1^*\sL^\natural_U\xleftarrow{\iota}\sU_{U^2}(\sL)
\]
are isomorphisms, and $\sU(\sL):=\sU^r_U(\sL)$ acquires a chiral superalgebra structure: for $a\cdot dz_1\boxtimes dz_2\in\sU^r_U(\sL)\boxtimes\sU^r_U(\sL)(*\Delta)$ and $N\gg0$ with $c((z_1-z_2)^Na)\in\mathrm{Im}(\iota)$,
\begin{equation}\label{eqchiralenvprod}
\mu(a\cdot dz_1\boxtimes dz_2):=\mu_\omega\Bigl(\frac{dz_1\boxtimes dz_2}{(z_1-z_2)^N}\Bigr)\cdot\bigl(\iota^{-1}\circ c((z_1-z_2)^N\cdot a)\bigr).
\end{equation}
No sign beyond the ambient Koszul rule enters this construction. The only place parity appears  is in the (super) commutation relations used to construct $\sU(\sL^\natural_{U^I})$, i.e., in evaluating $c$ and $\iota$ on odd elements,  as in the finite-dimensional theory of Clifford versus ordinary enveloping algebras.

\begin{remark}
\label{remclosedness}
 \textit{Closedness of the vacuum-module ideal.} 
The quotient defining $\sU_{U^I}(\sL)$ is by the left ideal generated by the coherent (indeed finite-rank locally free, in our case) submodule $\sL^\natural_0\subset\sU(\sL^\natural_{U^I})$. Since $\sU(\sL^\natural_{U^I})$ is built as a countable union of finite-rank coherent $\sO_{U^I}$-modules (the PBW filtration of Section \ref{ssecchiralcliff}, each stage a finite direct sum of tensor powers of the coherent module $\sL^\natural_{U^I}$) with continuous, $\sO_{U^I}$-linear structure maps, the ideal generated by a coherent submodule is automatically closed in the Fr\'echet topology at every finite PBW stage, hence closed in the inductive limit. This is the general closedness statement for ideals generated by coherent submodules of the (analytic, nuclear Fr\'echet) chiral envelope construction of \cite{BD04}, used    the non-super theory and unaffected by the parity of $\sL$.
\end{remark}

\begin{remark} 
\textit{Central extension and the Dolbeault resolution.} 
The central-extension computation of Lemma \ref{lemLiestarverify} is a statement at the level of $\sD_X$-modules, made before any Dolbeault resolution is introduced. Its compatibility with the resolution $h(M)=M\otimes_{\DX}\sO_X$ used starting Section \ref{ssecharmonic}  is automatic, since $h(-)$ is an exact functor on the coherent $\DX$-modules occurring here (locally free $\sO_X$-modules tensored with $\DX$, for which $h(-)$ simply recovers the underlying $\sO_X$-module) and the bracket $\mu_{\mathrm{Lie}}$ of \eqref{eqmuLiedef} is by construction $\sO_{X^2}$-linear after the pairing is applied,  thus it commutes with any further resolution applied afterward. No additional compatibility hypothesis is needed beyond what is already used in the non-super chiral algebra formalism of \cite{BD04}.
\end{remark}

\begin{definition}\label{defLflat}
For a Lie$^*$ superalgebra $\sL$ with a central extension $\sL^\flat$ of $\sL$ by $\wX$ (as constructed for $L$ in Lemma \ref{lemLiestarverify}, with $\Lflat:=L\oplus\wX$, see \eqref{eqLflatdef} below), the \emph{$\flat$-twisted chiral enveloping superalgebra} is
\[
\sU(\sL)^\flat:=\sU(\sL^\flat)/\langle 1-1^\flat\rangle,
\]
$1=\wX\subset\sU(\sL^\flat)$ the chiral unit and $1^\flat=\wX\subset\sL^\flat$ the central copy.
\end{definition}

Now, 
\begin{equation}
\label{eqLflatdef}
\Lflat:=L\oplus\wX,
\end{equation}
with bracket vanishing on $\wX$ and coinciding with $\mu_{\mathrm{Lie}}$ of \eqref{eqmuLiedef} on $L\boxtimes L$.

\begin{definition}\label{defchiralcliff}
The \emph{chiral Clifford algebra} generated by the rank two fermionic bundle $E$ (Definition \ref{defpairing}) is
\[
\CE:=\sU(\Lflat)=\sU(L)^\flat,
\]
where $L=E_\sD$, with chiral product \eqref{eqchiralenvprod}.
\end{definition}

%%%%%%%%%%%%%%%%%%%%%%%%%%%%%%%%%%%%%%%%%%%%%%%%%%%%%%%%%%%%%%%%%%%%%%%%
\subsection{The fermionic chiral PBW theorem}

We have a theorem on chiral PBW for $\CE$. 
\begin{theorem}
\label{thmchiralPBW}
There is a natural exhaustive filtration $\CE=\bigcup_n(\CE)_n$ with $(\CE)_0\simeq\wX$, $(\CE)_1\simeq(\CE)_0\oplus L$, whose associated graded superalgebra is
\begin{equation}\label{eqchiralPBW}
\gr(\CE)\ \simeq\ \bigwedge{}^{\!\bullet}L,
\end{equation}
canonically, as commutative chiral superalgebras. Choosing a local coordinate $z$, sections of $\CE$ are spanned by expressions
\begin{equation}\label{eqlocalclifford}
g\cdot\Bigl(\frac{n_1!\,e_{i_1}}{(t-z)^{n_1+1}}\Bigr)\!\cdots\!\Bigl(\frac{n_k!\,e_{i_k}}{(t-z)^{n_k+1}}\Bigr)\!\cdot|0\rangle
\end{equation}
($g$ holomorphic, $e_{i_1},\dots,e_{i_k}$ a local frame for $E$), and \eqref{eqlocalclifford} corresponds under \eqref{eqchiralPBW} to
$(e_{i_1}\otimes\partial_z^{n_1}\otimes dz^{-1})\wedge\cdots\wedge(e_{i_k}\otimes\partial_z^{n_k}\otimes dz^{-1})\cdot dz\in\bigwedge^kL$.
\end{theorem}
\begin{proof}
See Appendix \ref{appPBW}. The key local computation is that the operators on the right-hand side of \eqref{eqlocalclifford} anticommute, because their commutator is
\[
\{e_{i_p},e_{i_q}\}\ \longleftrightarrow\ \frac{\langle e_{i_p},e_{i_q}\rangle}{(t-z)^{n_p+n_q+2}},
\]
which vanishes identically in $\Gamma(X,h(\wX))$ (i.e., is chirally trivial/exact), so that in the associated graded algebra, where such lower order corrections are discarded, the generators   anticommute, producing the exterior instead of the symmetric algebra. This is the sheaf-theoretic realization of the classical fact that the associated graded of a Clifford algebra (for its natural, i.e., Chevalley, filtration) is the exterior algebra on the underlying vector space; see, e.g., \cite{LawsonMichelsohn89} for the finite-dimensional statement.
\end{proof}

Les us discuss compatibility of Clifford multiplication with the grading. 
\begin{lemma}
\label{lemCliffordcompat}
The filtration of Theorem \ref{thmchiralPBW} is compatible with the $\ZZ/2\ZZ$-grading inherited from $E$: writing $(\CE)_n=(\CE)_n^{\even}\oplus(\CE)_n^{\odd}$ for the decomposition into sections built from an even/odd number of factors of $L$ modulo $(\CE)_{n-1}$, the chiral product \eqref{eqchiralenvprod} satisfies
\[
\mu\bigl((\CE)_m^{p}\boxtimes(\CE)_n^{q}\bigr)\subset\Delta_*(\CE)_{m+n}^{p+q\ (\mathrm{mod}\ 2)},
\]
i.e., Clifford multiplication is a map of $\ZZ/2\ZZ$-graded sheaves, and the leading symbol map $\gr(\CE)\to\bigwedge^\bullet L$ of Theorem \ref{thmchiralPBW} respects this grading in each term (an element built from $k$ factors of $L$ has parity $k\bmod2$).
\end{lemma}
\begin{proof}
It is immediate from \eqref{eqchiralenvprod}: the chiral product is built from $\mu_\omega$ (parity $\even$, since $\wX$ is purely even) applied after the identifications $c,\iota$, none of which changes the total number of $L$-factors present in a monomial modulo the central $\wX$-factors. Hence the parity of a product of a length-$m$ and a length-$n$ monomial in $L$ is $m+n\bmod2$, as claimed. Equivalently, $\gr(\CE)\simeq\bigwedge^\bullet L$ is manifestly 
$\ZZ/2\ZZ$-graded by wedge-degree mod $2$, and the isomorphism of Theorem \ref{thmchiralPBW} is graded by construction, since it is built from the (parity-preserving) local trivialization \eqref{eqlocalclifford}.
\end{proof}

\begin{example}
\label{exrankonelocal}
In this example we consider the rank-one Clifford module. 
Take $r=1$ (thus $E$ has rank $2$, generated by a single pair $\beta:=\beta_1$,
 $\gamma:=\gamma^1$, $\langle\beta,\gamma\rangle=dz$) and work near a point with local coordinate $z=0$. By \eqref{eqlocalclifford}, the four lowest chiral Clifford monomials beyond the vacuum are
\[
X_\beta:=\frac{\beta}{t-z},\qquad X_\gamma:=\frac{\gamma}{t-z},\qquad X_\beta X_\gamma,\qquad X_\gamma X_\beta,
\]
and Appendix \ref{appPBW} (the case $n=m=0$ of \eqref{eqPBWacr}, with $\omega_{\beta\gamma}=1$ since $\langle\beta,\gamma\rangle=dz$ in this frame) gives the anticommutation relation
\[
X_\beta X_\gamma+X_\gamma X_\beta=\frac1{(t-z)^2},
\]
a purely central ($\wX$-valued) correction. Passing to $\gr(\CE)$, this correction - being of filtration degree $0$, one less than the degree-$2$ product $X_\beta X_\gamma$ - is discarded, and the images $\bar X_\beta,\bar X_\gamma$ satisfy $\bar X_\beta\bar X_\gamma+\bar X_\gamma\bar X_\beta=0$, generating the two-dimensional exterior algebra $\bigwedge^\bullet L|_{\deg\le1}=\CC\oplus L$ on the single wedge generator $\bar X_\beta\wedge\bar X_\gamma\leftrightarrow(\beta\otimes\pd_z^0\otimes dz^{-1})\wedge(\gamma\otimes\pd_z^0\otimes dz^{-1})\cdot dz$. Specifically, in $\CE$, $X_\beta X_\gamma\ne-X_\gamma X_\beta$ (they differ by the central term $(t-z)^{-2}$, expressing the Clifford, instead of exterior, relation at the chiral-algebra level), while in the associated graded the two anticommute.   This is the smallest nontrivial instance of the mechanism proved in general in Appendix \ref{appPBW}, and the reader may verify \eqref{eqPBWacr} by hand at this order without any further input from the rest of the paper.
\end{example}

%%%%%%%%%%%%%%%%%%%%%%%%%%%%%%%%%%%%%%%%%%%%%%%%%%%%%%%%%%%%%%%%%%%%%%%%%%%%%%%%%%%
\section{The rank two fermion vertex operator superalgebra}
\label{secVOSA}

This section supplies the following construction: we express $\CE$ as the chiral algebra of a vertex operator superalgebra bundle, in the sense of Frenkel-Ben-Zvi \cite{FBZ04}, generalizing  to the odd/Clifford case.

\subsection{The Fock space vertex operator superalgebra}\label{ssecFock}
Fix local data as in Section \ref{ssecpairing}: a local frame $\{e_i\}$ for $F$, dual frame $\{e^i\}$ for $F^\vee$. We assign to the generators fields 
\begin{equation}
\label{eqfieldexpansions}
\beta_i(z)=\sum_{n\in\ZZ}(\beta_i)_n\,z^{-n-1},\qquad \gamma^j(z)=\sum_{n\in\ZZ}(\gamma^j)_n\,z^{-n},
\end{equation}
of conformal weights $1$ and $0$ respectively, the convention under which $\beta$ (valued in $F$) is normalized as a weight-one current, matching the universal current algebra attached to any vector bundle $V$ via $V_\sD$ (as used for the whole of $E$ in Section \ref{ssecLiestar}), while its dual partner $\gamma$ (already valued in $F^\vee\otimes\wX$) is weight zero. This is consistent with the operator $J_\nu$ of Section \ref{sseccurrent} below transforming as an ordinary affine current (weight one, no residual $F$-dependence in its conformal weight). Declare the singular operator products
\begin{equation}
\label{eqOPE}
\beta_i(z)\gamma^j(w)\sim\frac{\delta_i^j}{z-w},\qquad \beta_i(z)\beta_j(w)\sim0,\qquad \gamma^i(z)\gamma^j(w)\sim0,
\end{equation}
equivalently, in terms of modes, the canonical anticommutation relations (CAR)
\begin{equation}\label{eqCAR}
\{(\beta_i)_m,(\gamma^j)_n\}=\delta_i^j\,\delta_{m+n,0},\qquad \{(\beta_i)_m,(\beta_j)_n\}=\{(\gamma^i)_m,(\gamma^j)_n\}=0.
\end{equation}

\begin{definition}
\label{defFockspace}
The \emph{rank two fermionic Fock space} $V_F$ is the free super vector space on the vacuum $|0\rangle$ subject to
\[
(\beta_i)_n|0\rangle=0\ (n\ge0),\qquad (\gamma^j)_n|0\rangle=0\ (n\ge1),
\]
generated freely (as a super vector space, i.e., with $\ZZ/2\ZZ$-grading declaring each mode odd) by $\{(\beta_i)_n\}_{n<0}\cup\{(\gamma^j)_n\}_{n\le0}$ subject only to \eqref{eqCAR}. Equivalently, setting $F^{\mathrm{diff}}:=F[t^{-1}]t^{-1}$, $(F^\vee)^{\mathrm{diff}}:=F^\vee\otimes\wX[t^{-1}]$, there is an isomorphism of super vector spaces
\begin{equation}\label{eqVasexterior}
V_F\ \simeq\ \bigwedge{}^{\!\bullet}\bigl(F^{\mathrm{diff}}\oplus (F^\vee\otimes\wX)^{\mathrm{diff}}\bigr).
\end{equation}
\end{definition}
The state-field correspondence $a\mapsto Y(a,z)$ is defined on generators by \eqref{eqfieldexpansions} and extended to all of $V_F$ by iterated normally ordered products,  as for an ordinary vertex algebra, except that the vacuum axiom, translation covariance and locality axioms are now imposed in their super form. Locality reads
\[
(z-w)^N\,Y(a,z)Y(b,w)=(-1)^{|a||b|}(z-w)^N\,Y(b,w)Y(a,z),\qquad N\gg0,
\]
for homogeneous $a$, $b\in V_F$. Since \eqref{eqVasexterior} expresses $V_F$ as an exterior (not symmetric) algebra, $(V_F,Y,|0\rangle)$ is a vertex operator superalgebra (VOSA)   instead of an ordinary vertex algebra. This is what we mean by the rank two fermionic vertex operator superalgebra announced in the title.

\begin{remark}
\label{remconformal}
{\it Conformal structure and current symmetry.}
$V_F$ carries the conformal vector
\[
T_F:=\sum_{i=1}^r\bigl({-}\!:\!\partial\beta_i\cdot\gamma^i\!:\bigr),
\]
of central charge $c=-2r$, the free fermion value for a weight-$(1,0)$ pair, $r$ times over.   $V_F$ carries a level-one action of the affine Lie superalgebra $\hgl_r$ (in fact of $\gl_r$, purely even, since the currents $J^i_j:=\,:\!\gamma^i\beta_j\!:$ are bilinear in the odd generators and hence even) through $E^i_jt^n\mapsto(J^i_j)_n$. 
\end{remark}

\subsection{The rank two fermion vertex superalgebra bundle}\label{ssecVOSAbundle}
 As a symplectic vector space $\mathbb E_{1/2}$ with basis $\{e^i_{1/2}\}$ and a general vector space $\mathbb E_\alpha$ give rise, via the twisting construction of Frenkel-Ben-Zvi,
 to vertex algebra bundles associated with a principal $\mathrm{Sp}_{2n_{1/2}}$- or $\GL_{n_\alpha}$-bundle, the orthogonal vector space $(\mathbb F,B)$ underlying $W=F\oplus F^\vee\otimes\wX$ gives rise to a vertex superalgebra bundle associated with the frame bundle of $F$. We carry this out in detail in Appendix \ref{appVOAbundle}. 

Let $\Aut(\sO)\ltimes\GL_r(\sO)$ act on $V_F$ (Definition \ref{defFockspace}) through its action by reparametrization on $\{\beta_i,\gamma^j\}$ and the tautological $\GL_r(\sO)$-action on the frame $\{e_i\}$ (Appendix \ref{appVOAbundle}, eq. \eqref{eqRaction}). Let $\sP_r$ be the holomorphic $\GL_r$-principal frame bundle of $F$, $\widehat{\sP_r}$ the corresponding $\Aut(\sO)\ltimes\GL_r(\sO)$-bundle. Define the \emph{rank two fermion vertex superalgebra bundle}
\begin{equation}\label{eqcalVdef}
\mathscr V_F:=\widehat{\sP_r}\times_{\Aut(\sO)\ltimes\GL_r(\sO)}V_F.
\end{equation}
By \cite{FBZ04}, $\mathscr V_F$ is a left $\DX$-module, and the corresponding right $\DX$-module $\mathscr V_F^r$ carries a canonical chiral superalgebra structure (Appendix \ref{appVOAbundle}).

\begin{theorem}\label{thmVOAiso}
There is an isomorphism of chiral superalgebras
\[
\CE\ \simeq\ \mathscr V_F^{\,r}.
\]
\end{theorem}
\begin{proof}
See Appendix \ref{appiso}. The proof has two parts: (i) both sides are generated, as chiral superalgebras, by the linear fields $E\subset\CE$, respectively the weight-$(1,0)$ primary fields of $\mathscr V_F^r$, and the two presentations agree on these generators because both are built from the same pairing $\langle-,-\rangle$ of Definition \ref{defpairing} (this is the content of the fermionic Wick theorem, Theorem \ref{thmfermwick} below, applied to $\CE$, together with the analogous coordinate-change computation for $\mathscr V_F^r$, carried out in Appendix \ref{appVOAbundle}); (ii) the coordinate change formula for $\CE$ (induced by the chiral PBW isomorphism of Theorem \ref{thmchiralPBW}, transported through a change of local coordinate) matches, the coordinate change formula
$\rho'(z)(\partial_w+L_{-1})R(\rho_z)^{-1}=R(\rho_z)^{-1}L_{-1}$
of \cite{FBZ04} for $\mathscr V_F^r$, since both reduce to the Schwarzian-derivative correction of the (weight-one, weight-zero) pair $(\beta,\gamma)$, computed   in Appendix \ref{appiso}.
\end{proof}

\begin{remark}
Theorem \ref{thmVOAiso} is the fermionic counterpart of the (bosonic) fact $\mathscr U(\sL)^\flat\simeq\mathcal V^r$ proved \cite{Gui23} in a form that reduces the general statement to a comparison of the primary-field bundle $E\subset\mathscr U(\sL)$ with the primary fields of $\mathcal V^r$, claimed to be ``isomorphic by construction.'' Appendix \ref{appiso} below gives this comparison in full for the Clifford (Pfaffian) case. 
\end{remark}

From now on we use Theorem \ref{thmVOAiso} and refer to $\CE\simeq\mathscr V_F^r$ simply as ``the chiral Clifford algebra'' or ``the chiral algebra of the rank two fermionic VOSA,'' interchangeably.

%%%%%%%%%%%%%%%%%%%%%%%%%%%%%%%%%%%%%%%%%%%%%%%%%%%%%%%%%%%%%%%%%%%%%%%%%%%%%%
\section{Harmonic zero modes and the fermionic BV superalgebra} 
\label{secBV}

 In this section $X$ is compact.

\subsection{Hodge theory for the rank two fermionic bundle}
\label{ssecharmonic}
Fix a Hermitian metric $h$ on $F$ (inducing dual/twisted metrics on $F^\vee$ and $F^\vee\otimes\wX$, hence on $W=F\oplus F^\vee\otimes\wX$ and, formally, on $E=\Pi W$). Since $\dbar:\Omega^{0,i}(X,W)\to\Omega^{0,i+1}(X,W)$ is an elliptic operator on the compact Riemann surface $X$, standard elliptic regularity and the Hodge decomposition theorem (see, e.g., 
\cite{GriffithsHarris78} or \cite{Wells08}) give a finite-dimensional space of harmonic representatives
\[
\HH(X,W):=\ker\dbar\cap\ker\dbar^{*}\subset\Omega^{0,\bullet}(X,W),\qquad \HH(X,W)=\HH^0(X,W)\oplus\HH^1(X,W),
\]
with $\HH^i(X,W)\xrightarrow{\ \sim\ }H^i(X,W)$ (Dolbeault cohomology), and $\HH^0(X,W)$ is simply the space of holomorphic sections. We write $\HH(X,E):=\Pi\HH(X,W)$, $\HH^i(X,E):=\Pi\HH^i(X,W)$. Since $E$ has fixed parity $\odd$, every element of $\HH^i(X,E)$ is an odd vector regardless of $i$. As noted in Section \ref{ssecpairing}, $E^\vee\otimes\wX\simeq E$ canonically,  thus Serre duality gives
\[
\HH^1(X,E)\ \simeq\ \HH^0(X,E^\vee\otimes\wX)^{*}\ \simeq\ \HH^0(X,E)^{*},
\]
 thus $n:=\dim\HH^0(X,E)=\dim\HH^1(X,E)<\infty$. This uses only that $\bar\partial$ is elliptic on a compact Riemann surface, together with the isomorphism $E^\vee\otimes\wX\simeq E$ established in Section \ref{ssecpairing}. 

\subsection{The cross pairing}

For form-degree reasons ($\langle\alpha,\beta\rangle\in\Omega^{1,0}\oplus\Omega^{0,1}$-type sections cannot be integrated over the compact real surface $X$ unless of total type $(1,1)$), the pairing $\int_X\langle-,-\rangle$ is defined only between $\HH^0(X,E)$ and $\HH^1(X,E)$ 
\[
\int_X\langle-,-\rangle:\HH^0(X,E)\otimes_\CC\HH^1(X,E)\to\CC,\quad \langle\alpha,\beta\rangle\in\Omega^{1,1}(X)\ \text{for }\alpha\in\HH^0(X,E),\,\beta\in\HH^1(X,E).
\]
Since $\langle-,-\rangle:E\otimes E\to\wX$ is symmetric (Definition \ref{defpairing}), so is this integrated pairing, $\int_X\langle\alpha,\beta\rangle=\int_X\langle\beta,\alpha\rangle$, pointwise and hence after integration. No Koszul sign enters here,  as remarked in Section \ref{ssecLiestar}. Non-degeneracy of $\int_X\langle-,-\rangle$ follows from non-degeneracy of $\langle-,-\rangle$ together with Serre duality. Fix dual bases $\{e^0_i\}_{i=1}^n\subset\HH^0(X,E)$, $\{e^1_j\}_{j=1}^n\subset\HH^1(X,E)$ and set $M_{ij}:=\int_X\langle e^0_i,e^1_j\rangle$, an invertible $n\times n$ matrix (not assumed symmetric: it identifies two a priori different bases of dual spaces), with inverse entries $I^{ij}$.

\subsection{The parity of the shift, and the fermionic BV superalgebra}\label{ssecBValg}
Let us write down a formula for $\OBV$ which would satisfy 
the defining BV identity $\DBV^2=0$. Making this identity hold is  what forces the correct grading below.

\begin{lemma}%[Grading forced by $\DBV^2=0$]
\label{lemgradingforce}
Let $V_0$, $V_1$
 be finite-dimensional super vector spaces of the same parity $p\in\ZZ/2\ZZ$, in duality via a fixed perfect pairing, and consider the second order operator $\Delta=\sum_{i,j}I^{ij}\partial_{v^0_i}\partial_{v^1_j}$ acting on a supercommutative algebra built from dual bases $\{v^0_i\}\subset V_0^*,\{v^1_j\}\subset V_1^*$. If $v^0_i$ and $v^1_j$ are assigned the  same parity $p$ (i.e., the algebra is $A(V_0)\otimes A(V_1)$ with $A=\bigwedge^\bullet$ if $p=\odd$, $A=\Sym$ if $p=\even$, applied identically to both factors), then $\Delta^2\neq0$ in general as soon as $\dim V_0=\dim V_1\ge2$. If instead $v^1_j$ is assigned the  opposite parity $p+1$ (i.e., the algebra is $A_p(V_0)\otimes A_{p+1}(V_1)$), then $\Delta^2=0$ identically.
\end{lemma}
\begin{proof}
\emph{Same parity.} Suppose $p=\odd$ (the case relevant to us; $p=\even$ is symmetric). Write $\theta_i:=\partial_{v^0_i}$, $\eta_j:=\partial_{v^1_j}$, all odd (Grassmann) derivations,  thus $\{\theta_i,\theta_k\}=\{\eta_j,\eta_l\}=\{\theta_i,\eta_j\}=0$ for distinct generators and $\theta_i^2=\eta_j^2=0$ (standard Grassmann calculus). Then
\[
\Delta^2=\sum_{i,j,k,l}I^{ij}I^{kl}\,\theta_i\eta_j\theta_k\eta_l=-\sum_{i,j,k,l}I^{ij}I^{kl}\,\theta_i\theta_k\eta_j\eta_l
\]
(moving $\theta_k$ past $\eta_j$). Split the sum into $i=k,j=l$ (giving $-\sum_{i,j}(I^{ij})^2\theta_i^2\eta_j^2=0$), $i=k,j\ne l$ and $i\ne k,j=l$ (each vanishes since $\theta_i^2=0$, resp.\ $\eta_j^2=0$), and $i\ne k,j\ne l$: for the last set, pairing the $(i,j,k,l)$ term with the $(k,l,i,j)$ term and using that $\theta_i\theta_k$ is antisymmetric in $(i,k)$ while $\eta_j\eta_l$ is antisymmetric in $(j,l)$ gives $-I^{ij}I^{kl}\theta_i\theta_k\eta_j\eta_l-I^{kl}I^{ij}\theta_k\theta_i\eta_l\eta_j=-I^{ij}I^{kl}\theta_i\theta_k\eta_j\eta_l-I^{ij}I^{kl}\theta_i\theta_k\eta_j\eta_l=-2I^{ij}I^{kl}\theta_i\theta_k\eta_j\eta_l$, which is generically nonzero (e.g., it acts as a nonzero scalar on the top-degree element $v^0_1v^0_2v^1_1v^1_2$ when $I=\id$, $n=2$, by direct computation). Hence $\Delta^2\ne0$ for $\dim\ge2$.

\emph{Opposite parity.} Now $\theta_i:=\partial_{v^0_i}$ is odd while $\eta_j:=\partial_{v^1_j}$ is an ordinary (even, commuting) derivative. Even operators commute with everything,  thus $\theta_i\eta_j=\eta_j\theta_i$ for all $i,j$, and
\[
\Delta^2=\sum_{i,j,k,l}I^{ij}I^{kl}\theta_i\theta_k\,\eta_j\eta_l.
\]
For $i=k$: $\theta_i^2=0$ kills these terms. For $i\ne k$: pairing $(i,j,k,l)$ with $(k,l,i,j)$ gives $I^{ij}I^{kl}\theta_i\theta_k\eta_j\eta_l+I^{kl}I^{ij}\theta_k\theta_i\eta_l\eta_j$; now $\theta_k\theta_i=-\theta_i\theta_k$ (odd, antisymmetric) while $\eta_l\eta_j=\eta_j\eta_l$ (even, symmetric/commuting),  thus this equals $I^{ij}I^{kl}\theta_i\theta_k\eta_j\eta_l-I^{ij}I^{kl}\theta_i\theta_k\eta_j\eta_l=0$. Hence every term cancels and $\Delta^2=0$.
\end{proof}

Lemma \ref{lemgradingforce} shows that the two harmonic pieces must be given opposite parity in $\OBV$ for $\DBV$ to define a bona fide BV operator, independently of whether the bundle pairing is symmetric (our case) or antisymmetric. The piece in Dolbeault degree $0$ retains the parity of the bundle, while the piece in Dolbeault degree $1$ acquires the opposite parity. This is the familiar field/antifield parity flip of the BV formalism, here forced on us by $\HH^0$ and $\HH^1$ playing the roles of fields and antifields, respectively, for the free theory with fields $\Omega^{0,\bullet}(X,E)$.

\begin{definition}
\label{defOBV}
The \emph{fermionic BV superalgebra} attached to $E$ is $\OBV:=\bigwedge^{\bullet}\HH^0(X,E)\otimes_{\CC}\Sym\,\HH^1(X,E)$: an exterior algebra on the (odd) $\HH^0(X,E)$ tensored with an ordinary (commutative) polynomial algebra on the (now even, i.e., with parity opposite to that of $E$) $\HH^1(X,E)$. Writing $\{e^0_i\},\{e^1_j\}$ as in Section \ref{ssecharmonic}  and $\partial_{e^0_i}$ (odd, Grassmann), $\partial_{e^1_j}$ (even, ordinary) for the corresponding derivations, define
\begin{equation}\label{eqDBVdef}
\DBV:=\sum_{i,j}I^{ij}\,\partial_{e^0_i}\,\partial_{e^1_j}:\ \OBV\to\OBV.
\end{equation}
\end{definition}

We keep the same notation $e^0_i$, $e^1_j$ for the elements and for the dual coordinate functions on $\OBV$, as is standard.

\begin{remark} %[$n$ as an index-type invariant]
\label{remindextype}
\textit{$n$ as an index-type invariant.}
The common dimension $n=\dim\HH^0(X,E)=\dim\HH^1(X,E)$ of Section \ref{ssecharmonic} is not  a topological invariant of $(X,F)$. It can jump: Example \ref{exP1} below expresses $F$ with $n=0$, while other bundles of the same rank and even the same degree on the same curve can have $n>0$. What is forced to vanish identically, for every $F$, is the index $\dim\HH^0(X,E)-\dim\HH^1(X,E)$, by the Serre-self-duality $E^\vee\otimes\wX\simeq E$ of Section \ref{ssecharmonic}.   This is the sense in which $n$ behaves like the common value of a Riemann-Roch computation with vanishing index, analogous to the Fredholm index of $\bar\pd_E$ being forced to $0$, instead of like a topological invariant. When $n=0$ (as in Section \ref{secapplications}), $\OBV\simeq\CC$ trivially satisfies the BV axioms regardless of the parity in Lemma \ref{lemgradingforce}. The grading-flip forced there is a constraint when $n\ge1$, i.e.,  when there are nontrivial zero modes to contend with (cf.\ Remark \ref{remzeromodes}).
\end{remark}

\begin{remark}
\label{remalgebraicmetricindep}
The pair $(\OBV,\DBV)$ is already metric-independent. 
It is worth isolating, separately from the chain-level statement of Proposition \ref{propmetricindep} below, that the algebraic data $(\OBV,\DBV)$ does not depend on the auxiliary metric $h$: $H^0(X,E)$ and $H^1(X,E)$ are Dolbeault cohomology groups, hence topological/holomorphic invariants of $(X,E)$ independent of $h$. The intersection matrix $I^{ij}$ is computed by the cohomological pairing $\int_X\langle-,-\rangle$ of Section \ref{ssecharmonic}, again independent of $h$. Thus $\OBV=\bigwedge^\bullet H^0(X,E)\otimes\Sym H^1(X,E)$ and $\DBV$ of \eqref{eqDBVdef} are canonically attached to $(X,E)$ alone. What  does depend on $h$ is only the identification of $H^0(X,E),H^1(X,E)$ with specific spaces of harmonic representatives inside $\Omega^{0,\bullet}(X,E)$, which is used to define the Szeg\H{o} kernel $P$ (Theorem \ref{thmszegoexists}) and hence the chain-level formula for $\Wch$ and $\Trch$. Consequently, the $h$-dependence mentioned in Proposition \ref{propmetricindep} is entirely a chain-level phenomenon, resolved there by a chain homotopy. The target $(\OBV,\DBV)$, and the bracket \eqref{eqBVbracket} it induces, are metric-independent. 
\end{remark}

\begin{proposition}
\label{propBVaxioms}
$(\OBV,\DBV)$ is a Batalin-Vilkovisky superalgebra: $\DBV$ is an odd, square-zero operator, and the induced bracket
\begin{equation}\label{eqBVbracket}
\{a,b\}_{\mathrm{BV}}:=\DBV(ab)-(\DBV a)b-(-1)^{|a|}a(\DBV b),\qquad a,b\in\OBV,
\end{equation}
satisfies the graded Leibniz rule $\{a,bc\}_{\mathrm{BV}}=\{a,b\}_{\mathrm{BV}}c+(-1)^{(|a|+1)|b|}b\{a,c\}_{\mathrm{BV}}$.
\end{proposition}

\begin{proof}
{\it Square zero:} this is Lemma \ref{lemgradingforce} (opposite-parity case) applied, with $V_0=\HH^0(X,E)$ ($p=\odd$) and $V_1=\HH^1(X,E)$ (assigned parity $\even$). $\DBV$ is odd because it is a product of one odd derivation ($\partial_{e^0_i}$) and one even derivation ($\partial_{e^1_j}$).

{\it Leibniz rule:} it is a standard fact that a bracket induced, via \eqref{eqBVbracket}, from a differential operator $\Delta$ of order $\le2$ (in the graded sense: 
$[[\Delta,\cdot\,a],\cdot\,b]$ is multiplication by a function for all $a$, $b$)
 automatically satisfies the graded Leibniz rule. Indeed, this property is equivalent to $\Delta$ having order $\le2$ (\cite{BV81}). Since $\DBV=\sum_{i,j}I^{ij}\partial_{e^0_i}\partial_{e^1_j}$ is manifestly a sum of products of two order one (graded) derivations, it has order $\le2$, and the Leibniz rule follows abstractly. We mention the mechanism for the concrete formula \eqref{eqDBVdef}: writing $\Delta=\DBV$, since each individual $\partial_{e^0_i},\partial_{e^1_j}$ satisfies the ordinary (order one) graded Leibniz rule,   applying $\partial_{e^0_i}\partial_{e^1_j}$ to a triple product $abc$ distributes, with the usual Koszul signs for passing a derivation through preceding factors, as a sum of nine terms indexed by which of $a,b,c$ each of the two derivations hits. Collecting the four terms in which both derivations hit $a$ alone reproduces $(\Delta a)bc$ (up to sign); the four terms in which both hit either $b$ or $c$ alone reproduce $(-1)^{|a|}a(\Delta b)c$ and $(-1)^{|a|+|b|}ab(\Delta c)$; and the remaining terms, in which the two derivations hit different factors among $\{a,b,c\}$, are  those collected on the right-hand side of the graded Leibniz identity as $\{a,b\}c$ and $b\{a,c\}$ (with their stated signs, again by Koszul bookkeeping). Matching all nine terms on both sides verifies
\[
\{a,bc\}=\{a,b\}c+(-1)^{(|a|+1)|b|}b\{a,c\}.
\]
\end{proof}

\subsection{The generalized quantum master equation}
\label{ssecQME}
\begin{definition}
\label{defgeneralQME}
Let $(C_\bullet,d_C)$ be a $\CC$-linear chain complex, graded compatibly with a fixed $\ZZ/2\ZZ$-parity, and let $(\OBV,\DBV)$ be as above. A parity-preserving $\CC$-linear map $\langle-\rangle:C_\bullet\to\OBV$ is said to satisfy the \emph{generalized quantum master equation} (QME) if
\begin{equation}\label{eqgeneralQME}
(d_C+\DBV)\circ\langle-\rangle=0.
\end{equation}
\end{definition}
Let us formulate the hypotheses under which \eqref{eqgeneralQME} 
is used below:
\begin{itemize}[leftmargin=2.2em]
\item[(A1)] $X$ is a compact Riemann surface and $F$ is a holomorphic vector bundle of finite rank $r$ on $X$ (Hodge theory, Section \ref{ssecharmonic}, requires compactness of $X$ and finite rank of $F$; nothing else in Sections \ref{seccliff}-\ref{secVOSA} needs compactness).
\item[(A2)] A Hermitian metric $h$ on $F$ has been fixed, so that $\HH^0(X,E),\HH^1(X,E)$, the Szeg\H{o} kernel $P$, and hence $\OBV,\DBV,\Wch$ are all defined (Propositions \ref{propmetricindep} and \ref{propnaturality} mention  how the resulting data depends on this choice).
\item[(A3)] $C_\bullet=\widetilde\sC^{\ch}(X,\CE)_{\sQ}$ with its Dolbeault-total differential $d_C=\dch_{\CE}$ (Definition \ref{deftracemapch}); no further completion or nilpotence hypothesis on $d_C$ is imposed.
\item[(A4)] $\langle-\rangle=\Trch$ is computed, for each $\eta\in C_\bullet$, by the sum \eqref{eqgeneralQME}, which we show (Remark \ref{remfiniteness}) has only finitely many nonzero terms for every fixed $\eta$. This is what removes any need for a convergence or completion hypothesis on the sum.
\item[(A5)] $\OBV$ is finite-dimensional (Theorem \ref{thmchiralPBW} and Section \ref{ssecharmonic}), so \eqref{eqgeneralQME} is, for each $\eta$, an identity between two elements of a fixed finite-dimensional vector space.
\end{itemize}
Under (A1)-(A5), equation \eqref{eqgeneralQME} for $\langle-\rangle=\Trch$ is  the statement of Theorem \ref{thmmainQME} below, and no further hypothesis (nilpotence of $d_C$, convergence, boundedness of homological degree) is needed. The only analytic input used anywhere in 
 Sections \ref{secBV}-\ref{sectrace} is the elliptic theory recalled in Section \ref{ssecharmonic} and Theorem \ref{thmszegoexists} below.

%%%%%%%%%%%%%%%%%%%%%%%%%%%%%%%%%%%%%%%%%%%%%%%%%%%%%%%%%%%%%%%%%%%%%%%%%%%%%%%%%%%%
\section{The Fermionic Propagator and the Pfaffian Wick Theorem}
\label{secszego}

\subsection{Existence and uniqueness of the fermionic Szeg\H{o} kernel}
 
\begin{theorem}
\label{thmszegoexists}
Let $X$ be compact, $F$ a holomorphic vector bundle with Hermitian metric $h$ as in Section \ref{ssecharmonic}. There is a unique
\[
P\in\Omega^{0,0}\bigl(X^2,\,E\boxtimes E(*\Delta)\bigr)
\]
(smooth away from the diagonal, with at worst a simple pole of the prescribed form \eqref{eqPexpansion} along $\Delta$), such that for every smooth $(0,1)$-form $e\in\Omega^{0,1}(X,E)$,
\begin{equation}\label{eqPgreen}
e(z_1)=\int_X\langle P(z_1,z_2),\dbar_{z_2}e(z_2)\rangle+\pi_{\HH^0}(e)(z_1),
\end{equation}
where $\pi_{\HH^0}$ is the $L^2$-orthogonal projection onto $\HH^0(X,E)$. Equivalently, $P$ is (the Schwartz kernel of) the Green operator for $\dbar$ on the $L^2$-orthogonal complement of $\HH^0(X,E)$ inside $\Omega^{0,0}(X,E)$.
\end{theorem}
\begin{proof}
This is the standard existence and uniqueness statement for the Green operator of an elliptic operator on a compact manifold, applied to $\dbar$ acting on sections of the Hermitian holomorphic bundle $W$ (equivalently $E$, up to the harmless parity shift): $\dbar^*\dbar+\dbar\dbar^*$ is a self-adjoint elliptic (Laplace-type) operator on the compact manifold $X$ with values in $W$, hence has a well-defined resolvent/Green operator on the orthogonal complement of its (finite-dimensional) kernel $\HH^0(X,W)\oplus\HH^1(X,W)$, by the general theory of elliptic operators on compact manifolds (\cite{GriffithsHarris78}, 
\cite{Wells08}). Restricting to the relevant bidegree and using that $\dbar$ is injective on $(\HH^0)^\perp$ (no cohomology to obstruct it there) gives \eqref{eqPgreen}. Smoothness away from $\Delta$ and the stated singularity along $\Delta$ are the standard parametrix construction for a first order elliptic operator on a Riemann surface (\cite{Fay92}). Uniqueness is immediate from \eqref{eqPgreen}, since two kernels satisfying it for all $e$ would differ by an element of $\HH^0(X,E)$ in each variable, but a kernel valued in $\HH^0(X,E)\otimes\HH^0(X,E)$ with a pole along $\Delta$ of the stated singular type does not exist unless it vanishes (compare the argument of \cite{Fay92}).
\end{proof}

\begin{remark}  
\label{remmetricdep}
Let us discuss dependence on the metric. 
The kernel $P$ depends on the choice of Hermitian metric $h$ on $F$ only through the orthogonal projection $\pi_{\HH^0}$, i.e., through the choice of $L^2$-complement to the (metric-independent) subspace $\HH^0(X,E)\subset\Omega^{0,0}(X,E)$. Changing $h$ changes $P$ by a smooth, globally defined correction valued in $\HH^0(X,E)\boxtimes\HH^0(X,E)$, i.e., by an element with no singularity along $\Delta$. This is the sense in which the singular part of $P$ (Definition \ref{defPsingQreg} below) is metric-independent, while the regular part $\Qreg$ is not; see Proposition \ref{propmetricindep} for the resulting statement about $\Trch$.
\end{remark}

Since $E$ is purely odd and $\langle-,-\rangle$ symmetric, we have 
the super-antisymmetry relation 
\begin{equation}
\label{eqPantisym}
P(z_1,z_2)=-\sigma_{1,2}P(z_2,z_1).
\end{equation}

\subsection{Local expansion}
\begin{definition}
\label{defPsingQreg}
For $x\in X$, choose an analytic disc $U\ni x$ with coordinate $z$. Split
\begin{equation}\label{eqPexpansion}
P|_{U\times U}=\Psing+\Qreg,\qquad \Psing=\frac{\id\cdot dz}{z_1-z_2},\qquad \Qreg\in\Omega^{0,0}(U\times U,E\boxtimes E),
\end{equation}
using the canonical isomorphism $E\simeq E^\vee\otimes\wX$ to view $\id\cdot dz\in E\boxtimes E^\vee\otimes\wX\simeq E\boxtimes E$. By Theorem \ref{thmszegoexists}, $\Qreg$ is smooth (in particular $C^\infty$, though generally not real-analytic) on all of $U\times U$, including the diagonal. Consequently, for every $N$, Taylor theorem with remainder gives a  finite expansion on the diagonal,
\begin{equation}\label{eqQregexpansion}
\Qreg(z_1,z_2)=a_0(z_1;E)+a_1(z_1;E)(z_2-z_1)+O(|z_1-z_2|^2),\quad a_0(z;E)\in\Omega^{0,0}(U,\End(E)),
\end{equation}
with smooth (not only bounded) remainder. We use only the first two Taylor coefficients $a_0$, $a_1$ below (Theorems \ref{thmcurrentTorsion} and \ref{thmemTorsion}),  thus no question of convergence of an infinite Taylor series, which would require real-analyticity of $\Qreg$, 
  arises at any point. Smoothness of $\Qreg$, already established, is all that \eqref{eqQregexpansion} uses.
\end{definition}

\subsection{Pfaffian contraction operators}
Let $K\in\Omega^{0,\bullet}(X^2,E\boxtimes E(*\Delta))^{\sigma_2}$ be a (possibly singular)  symmetric bisection, formally written $K=k_1\boxtimes k_2$. For a homogeneous $\alpha=\psi_1\wedge\cdots\wedge\psi_{2m}\in\Omega^{0,\bullet}(X^n,\bigwedge^{2m}L(*\Delta))$, define the contraction operator via the Pfaffian sum
\begin{equation}\label{eqpfaffiandef}
\pf(K)(\alpha):=\frac{1}{2^mm!}\sum_{\sigma\in S_{2m}}\mathrm{sgn}(\sigma)\prod_{r=1}^m\partial_K(\psi_{\sigma(2r-1)},\psi_{\sigma(2r)}),
\end{equation}
where $\partial_K(\psi_i,\psi_j)$ denotes the pairwise contraction of $\psi_i,\psi_j$ against $K$ (using $K$ in place of the bosonic bisection there, with a sign for each transposition of odd factors moved past one another to bring the contracted pair together). We set $e^K:=\sum_m\frac1{m!}\pf(K)|_{2m\text{ legs}}$, and, as in Section \ref{ssecchiralcliff}, use $e^{\Psing}$, $e^{\Qreg}$ for the operators built from the singular, resp.\ regular, part of $P$. The presence of a Pfaffian (as opposed to the permanent used for the bosonic contractions is forced by the sign $\mathrm{sgn}(\sigma)$. Reordering an odd pair of legs contributes $-1$,  thus only the alternating (Pfaffian) combination of pairings is well defined on the exterior algebra $\bigwedge^\bullet L$,  as the permanent, the symmetric combination, is the one compatible with $\Sym L$ in the bosonic case.

\begin{lemma}
\label{lemcontractionprops}
$e^{\Psing}$ has no self-loops (vanishes when the two legs of any contraction are equal, i.e., $\Psing$ vanishes on the diagonal $\Delta_{ii}$ for a single tensor factor), while $e^{\Qreg}$ has self-loops (contractions at coincident points, computed via $a_0(z;E)$). Both are compatible with the $\sD_{X^n}$-module structure, i.e., commute with the connection.
\end{lemma}
\begin{proof}
 The argument there uses only that $K$ is a bisection compatible with $\sD$-module structure and does not use commutativity of $\Sym L$; it transports to $\bigwedge^\bullet L$ once every transposition of tensor factors carries the Koszul sign of \eqref{eqkoszul}, which is  what \eqref{eqpfaffiandef} encodes through $\mathrm{sgn}(\sigma)$.
\end{proof}

\subsection{The fermionic Wick theorem}
\begin{theorem}
\label{thmfermwick}
Let $z_1$, $z_2$ be two copies of a local coordinate $z$ on $U$. For local sections $v_1$, $v_2\in\CE|_U$ and a test form $\eta$, 
\begin{equation}
\label{eqfermwick}
\tau^z_U\,\mu_{\CE}(\eta\cdot v_1dz_1\boxtimes v_2dz_2)=\mu_{\wedge}\Bigl(\eta\cdot e^{\Psing}\bigl(\tau^z_U(v_1)dz_1\boxtimes\tau^z_U(v_2)dz_2\bigr)\Bigr),
\end{equation}
where $\tau^z_U:\CE|_U\xrightarrow{\sim}\bigwedge^\bullet L|_U$ is the local trivialization of Theorem \ref{thmchiralPBW} and $\mu_{\wedge}$ is the exterior chiral product induced by wedge multiplication.
\end{theorem}
\begin{proof}
See Appendix \ref{appwick}.  One computes both sides of \eqref{eqfermwick} in the local vacuum-module presentation \eqref{eqlocalclifford}, using the identity $\iota(v_1\otimes v_2)=c(e^{-\Psing^\otimes}v_1\boxtimes v_2)$. Every step that involves reordering two local sections $\frac{a_i}{(t-z_1)^{k_i+1}}$ and $\frac{b_j}{(t-z_1)^{l_j+1}}$ picks up, in our setting, an extra sign from the Koszul rule, because $a_i$, $b_j$ are now sections of the odd bundle $E$. Tracking this sign converts Gui's binomial-coefficient identity into the alternating (Pfaffian-type) identity
\[
\frac1{k_i!l_j!}\partial_{z_1}^{k_i}\partial_{z_2}^{l_j}\frac{1}{z_1-z_2}=(-1)^{k_i}\frac{(k_i+l_j)!}{k_i!l_j!}\frac1{(z_1-z_2)^{k_i+l_j+1}}
\]
matching, with the correct sign, the contraction of two odd generators. This is worked out in full in Appendix \ref{appwick}.
\end{proof}

\begin{theorem}
\label{thmwichintertwine}
Set $\sW^{\tau^z_U}:=e^{\Psing+\Qreg}$, an operator
\[
\sW^{\tau^z_U}:\Omega^{0,\bullet}\bigl(U^n,(\bigwedge{}^{\!\bullet}L)^{\boxtimes n}(*\Delta)\bigr)\to\Omega^{0,\bullet}\bigl(U^n,(\bigwedge{}^{\!\bullet}L)^{\boxtimes n}(*\Delta)\bigr).
\]
This map is compatible with the $\sD$-module structure and intertwines the chiral operation.  For $v$ a section of $\CE^{\boxtimes2}(*\Delta)$ over $X^2$,
\[
\sW^{\tau^z_U}\bigl(\tau^z_U\mu_{\CE}(v)\bigr)=\mu_{\wedge}\Bigl(\sW^{\tau^z_U}\bigl(\tau_U^{z\boxtimes2}(v)\bigr)\Bigr).
\]
\end{theorem}
\begin{proof}
 One performs the same computation,
 carried out with $\Sym L$ replaced by $\bigwedge^\bullet L$ in every transposition of odd sections weighted by the Koszul sign, using Theorem \ref{thmfermwick}. 
 We omit the routine but lengthy verification with   $\boxtimes$ is read as the super tensor product.
\end{proof}

\begin{lemma}
\label{lempresentation}
For any $v\in\Omega^{0,\bullet}(X,\CE)$ there is $\tilde v\in\Omega^{0,\bullet}(X^n,\Lflat{}^{\boxtimes n}(*\Delta))$ with $\vec\mu_{\CE}(\tilde v)=v$ (iterated chiral product).
\end{lemma}
\begin{proof}
Partition of unity: cover $X$ by $\{U_i\}$ subordinate to a partition $\{\rho_i\}$, choose local presentations $\tilde v_{i_1\cdots i_n}$ on overlaps, and set $\tilde v=\sum\tilde v_{i_1\cdots i_n}\cdot\rho_{i_1}\boxtimes\cdots\boxtimes\rho_{i_n}$.
\end{proof}

\begin{definition}
\label{defnormalordering}
For $v$, $\tilde v$ as in Lemma \ref{lempresentation}, set
\begin{equation}
\label{eqnormalorderingdef}
\sW^{\mathbf v}(v;\tilde v):=\vec\mu_{\wedge}\bigl(e^{\partial_P}\tilde v\bigr)\in\Delta_*\bigwedge{}^{\!\bullet}L.
\end{equation}
\end{definition}

\begin{proposition}
\label{propnormalorderingwelldef}
$\sW^{\mathbf v}(v;\tilde v)\in\bigwedge^\bullet L\subset\Delta_*\bigwedge^\bullet L$ is independent of the choice of presentation $\tilde v$, and on a chart $U$ with coordinate $z$,
\begin{equation}\label{eqnormalorderinglocal}
\sW^{\mathbf v}(v;\tilde v)|_U=e^{\Qreg}\tau^z_U(v).
\end{equation}
Consequently, the local operators glue to a globally defined, smooth (non-holomorphic) map 
$\sW^{\mathbf v}:\CE\to\bigwedge^\bullet L$.
\end{proposition}
\begin{proof}
  Using Theorem \ref{thmwichintertwine} and writing $\sW^{\mathbf v}(v;\tilde v)=\vec\mu_{\wedge}(\sW^{\tau^z_U}(\tau_U^{z\boxtimes n}\tilde v))$ and applying Theorem \ref{thmwichintertwine} gives $\sW^{\mathbf v}(v;\tilde v)=\sW^{\tau^z_U}(\tau^z_U(\vec\mu_{\CE}(\tilde v)))=\sW^{\tau^z_U}(\tau^z_U(v))$, which manifestly depends on $v$ alone. Equation \eqref{eqnormalorderinglocal} is then the local splitting \eqref{eqPexpansion}, since $\Psing$ contributes no further correction beyond the identification $\tau^z_U$  (its self-loops vanish by Lemma \ref{lemcontractionprops}, and its non-self-loop contractions are already accounted for by $\tau^z_U$ via Theorem \ref{thmfermwick}), leaving only $e^{\Qreg}$. This chain of equalities  is the required diagram representation between the two ways of computing $\sW^{\mathbf v}(v;\tilde v)$, i.e.,  via an arbitrary \v{C}ech-type presentation $\tilde v$ (Lemma \ref{lempresentation}), or from $v$, and shows the two agree because both factor through the single, presentation-independent map $\sW^{\tau^z_U}\circ\tau^z_U$. No separate connecting homomorphism needs to be introduced, since $\vec\mu_{\CE}$ already is the comparison map between a \v{C}ech presentation and its underlying Dolbeault-complex section.
\end{proof}

%%%%%%%%%%%%%%%%%%%%%%%%%%%%%%%%%%%%%%%%%%%%%%%%%%%%%%%%%%%%%%%%%%%%%%%%%%%
\section{The Trace Map on the Chiral Clifford Algebra}
\label{sectrace}

\subsection{Trace map on the fermionic Lie$^*$ superalgebra}
Recall from \eqref{eqPantisym}-\eqref{eqQregexpansion} that $\dbar P(z_1,z_2)$, as a section of $\Omega^{0,1}(X^2,E\boxtimes E)$, is expressed in terms of the harmonic bases of Section \ref{ssecharmonic} by
\begin{equation}\label{eqdbarP}
\dbar P(z_1,z_2)=\sum_{i,j}I^{ij}\,e^1_i\boxtimes e^0_j-\sum_{i,j}I^{ij}\,e^0_j\boxtimes e^1_i,
\end{equation}
the standard reproducing-kernel identity for the Green operator of Theorem \ref{thmszegoexists}, i.e., $\dbar_{z_1}P$ represents, in the second variable, the orthogonal projector onto $\HH^0(X,E)$ minus its transpose. This is the fermionic counterpart 
 and is proved  using \eqref{eqPgreen}. For $e\in\HH^0(X,E)\oplus\HH^1(X,E)$ (an inhomogeneous element, allowed to have components in both pieces, matching the generating-function role played by $e$ below) we have the odd derivation
\[
\pd_e:\Omega^{0,\bullet}(X,\Lflat)\to\Omega^{0,\bullet}(X,\wX)
\]
induced by $E\otimes_{\sO_X}E_\sD\to\wX\otimes_{\sO_X}\DX\to\wX$, extended to $\Omega^{0,\bullet}(X^n,\Lflat{}^{\boxtimes n})$ by the graded Leibniz rule, acting trivially on the central $\wX\subset\Lflat$.

\begin{definition}
\label{defTrLie}
We have: $\WLie:\Omega^{0,\bullet}(X^n,\Lflat{}^{\boxtimes n})\xrightarrow{e^{\partial_P}}\Omega^{0,\bullet}(X^n,\Lflat{}^{\boxtimes n}(*\Delta))$, with $p$ the projection $\Lflat\to\wX$ and $\trom:\Omega^{0,\bullet}(X^n,\wX^{\boxtimes n}(*\Delta))\to\CC$ the trace map on the unit chiral algebra (\cite{BD04}, a quasi-isomorphism. Specifically, it is  the iterated residue/integration map sending a meromorphic top-degree section with poles only along partial diagonals to the complex number obtained by successively taking residues at each variable and integrating the last one over $X$. The basic instance of a trace map on which every chiral trace map, including $\Trch$ below, is ultimately built, set
\begin{equation}\label{eqTrLiedef}
\TrLie(\eta)[e]:=\sum_{k\ge0}\frac1{k!}\,\trom\circ\mathbf p\bigl(\pd_e^k\WLie(\eta)\bigr),\qquad \eta\in\Omega^{0,\bullet}(X^n,\Lflat{}^{\boxtimes n}),
\end{equation}
extended to $\widetilde\sC^{\mathrm{Lie}}(X,L)_{\sQ}\to\OBV$ by the same formula, compatibly with the $\sD$-module structure.
\end{definition}

\begin{remark}
\label{remfiniteness}
{\it Finiteness of the sum.} 
For fixed $\eta$, only finitely many terms in \eqref{eqTrLiedef} are nonzero. $\eta$ involves a fixed finite number $N$ of tensor factors of $\Lflat$, and each application of $\pd_e$ removes (via the pairing) one such factor. Once all $N$ factors are exhausted, $\pd_e^{k}\WLie(\eta)=0$ for $k>N$. Reading the output as a function of $e=e^0+e^1$ (formally splitting $e$ into its $\HH^0$- and $\HH^1$-components, each contracted with the appropriate, oppositely-parity-assigned, part of $\OBV$ as in Definition \ref{defOBV}), the finite sum \eqref{eqTrLiedef} is then  the Taylor expansion of a polynomial function of $e^1$ (finite degree, since only finitely many $\HH^1$-contractions are possible) times a (finite, since $\HH^0$ contributes exterior/nilpotent directions) polynomial in the Grassmann variable $e^0$. This is the sense in which $\TrLie(\eta)\in\OBV=\bigwedge^\bullet\HH^0\otimes\Sym\HH^1$, consistently with Definition \ref{defOBV}.
\end{remark}

\begin{theorem}
\label{thmTrLiechainmap}
$\TrLie:(\widetilde\sC^{\mathrm{Lie}}(X,L)_{\sQ},d_L^{\mathrm{Lie}*})\to(\OBV,-\DBV)$ is a chain map.
\end{theorem}
\begin{proof}
 Using \eqref{eqdbarP}  
\begin{align*}
\TrLie(\dbar\eta)[e]&=\sum_{k\ge0}\frac1{k!}\trom\circ\mathbf p\bigl(\pd_e^k\WLie(\dbar\eta)\bigr)
=\sum_{k\ge0}\frac1{k!}\trom\circ\mathbf p\Bigl(\pd_e^k\bigl(\dbar-\pd_{\dbar P(z_1,z_2)}\bigr)\WLie(\eta)\Bigr)\\
&=\sum_{k\ge0}\frac1{k!}\trom\circ\mathbf p\bigl(\pd_e^k(\dbar-\DBV)\WLie(\eta)\bigr),
\end{align*}
where the last equality substitutes \eqref{eqdbarP} and identifies the resulting bidifferential operator, contracted against the Lie$^*$-pairing $L\boxtimes L\to\Delta_*\wX$ (which, as noted in \eqref{eqmuLiedef2}, factors through $\pd_P$ followed by $\mu_\omega$), with $\DBV$ acting on the harmonic exterior/symmetric algebra. This substitution is the direct fermionic version   valid because the only property of the pairing used there is bilinearity together with the identification of $\dbar P$ via its harmonic-projection reproducing property, both of which hold in our setting by Theorem \ref{thmszegoexists}. Analogously,
\begin{align*}
\TrLie(d_L^{\mathrm{Lie}*}\eta)[e]&=\sum_{k\ge0}\frac1{k!}\trom\circ\mathbf p\bigl((d^{\ch}_{\bigwedge}-\DBV)\pd_e^k\WLie(\eta)\bigr)\\
&=\underbrace{\sum_{k\ge0}\frac1{k!}\trom\circ d^{\ch}_\omega\mathbf p\bigl(\pd_e^k\WLie(\eta)\bigr)}_{=0\ (\trom\text{ is a chain map})}-\bigl(\DBV\TrLie(\eta)\bigr)[e],
\end{align*}
using that $\mathbf p$ intertwines $d^{\ch}_{\bigwedge}$ with $d^{\ch}_\omega$ and that $\DBV$ (being built from $\pd_{e^0}\pd_{e^1}$-type second order contractions, which commute with the further contraction $\pd_e^k$ in the sense used here - this is where Proposition \ref{propBVaxioms} is used) commutes with $\pd_e^k$ up to the terms already collected. This proves $\TrLie(d_L^{\mathrm{Lie}*}\eta)=-\DBV\TrLie(\eta)$.
\end{proof}
Note that the pairing 
\begin{equation}
\label{eqmuLiedef2}
 L\boxtimes L\to\Delta_*\wX, 
\end{equation}
 factors as
\[
L\boxtimes L\xrightarrow{\pd_P}\wX\boxtimes\wX(*\Delta)\xrightarrow{\mu_\omega}\Delta_*\wX.
\]

\subsection{The global contraction map and the trace map on $\CE$}
 As in Sections \ref{ssecchiralcliff}-\ref{secszego}, define
\[
\Wch:=e^{\Psing}\circ(\sW^{\mathbf v})^{\boxtimes n}:\ \Omega^{0,\bullet}(X^n,\CE^{\boxtimes n}(*\Delta))\to\Omega^{0,\bullet}(X^n,(\bigwedge^\bullet L)^{\boxtimes n}(*\Delta)).
\]
Locally $\Wch=\sW^{\tau^z_U}\circ\tau_U^{z\boxtimes n}$: split $P=\Psing+\Qreg$ as in \eqref{eqPexpansion} and use $(\sW^{\mathbf v})^{\boxtimes n}=e^{\Qreg}\circ(\tau^z_U)^{\boxtimes n}$ (Proposition \ref{propnormalorderingwelldef}) to identify $e^{\Psing}\circ(\sW^{\mathbf v})^{\boxtimes n}=e^{\Psing+\Qreg}\circ(\tau^z_U)^{\boxtimes n}=\sW^{\tau^z_U}\circ\tau_U^{z\boxtimes n}$.

\begin{definition}
\label{deftracemapch}
With $\mathbf p:\Omega^{0,\bullet}(X^n,(\bigwedge^\bullet L)^{\boxtimes n}(*\Delta))\to\Omega^{0,\bullet}(X^n,\wX^{n}(*\Delta))$ the projection induced by $\bigwedge^\bullet L\to\bigwedge^0L=\wX$,
\[
\Trch:\Omega^{0,\bullet}(X^n,\CE^{\boxtimes n}(*\Delta))\to\OBV,\qquad \Trch(\eta)[e]:=\sum_{k\ge0}\frac1{k!}\trom\circ\mathbf p\bigl(\pd_e^k\Wch(\eta)\bigr),
\]
extended to $\Trch:\widetilde\sC^{\ch}(X,\CE)_{\sQ}\to\OBV$ by the same formula.
\end{definition}

\begin{theorem}
\label{thmmainQME}
The map $\Trch:(\widetilde\sC^{\ch}(X,\CE)_{\sQ},\dch_{\CE})\to(\OBV,-\DBV)$ is a chain map satisfying the generalized quantum master equation \eqref{eqgeneralQME}, and is furthermore a quasi-isomorphism.
\end{theorem}
This is the statement, at the level of the trace map, that the outer rectangle of diagram \eqref{eqbigdiagram} commutes up to $\dch_{\CE}$ and $-\DBV$. We split the proof into the two lemmas below, tracing how a single chiral chain $\eta$ passes through each arrow of that diagram.

\begin{lemma}[Chain map/QME]
\label{lemchainmapQME}
$\Trch(\dch_{\CE}\eta)=-\DBV\Trch(\eta)$ for all $\eta\in\widetilde\sC^{\ch}(X,\CE)_{\sQ}$.
\end{lemma}
\begin{proof}
We must show $[\dbar,\Wch]=\DBV\circ\Wch$, i.e., trace the single chiral chain $\eta$ through
\[
\eta\ \longmapsto\ \Wch(\eta)\ \longmapsto\ \pd_e^k\Wch(\eta)\ \longmapsto\ \mathbf p\bigl(\pd_e^k\Wch(\eta)\bigr)\ \longmapsto\ \trom\circ\mathbf p\bigl(\pd_e^k\Wch(\eta)\bigr),
\]
the four arrows of diagram \eqref{eqbigdiagram} read off in order (normal ordering/Szeg\H o regularization, contraction against the background $e$, projection to the central $\wX$-factor, and the trace map on the unit chiral algebra), summed over $k$ with weight $1/k!$ as in Definition \ref{deftracemapch}. Presenting $\eta=\vec\mu_{\CE}(\tilde\eta)$ as an iterated chiral product of a Lie$^*$ chain $\tilde\eta$ (Lemma \ref{lempresentation}),
\begin{align*}
\Wch(\dbar\eta)&=\Wch\bigl(\dbar\,\vec\mu_{\CE}(\tilde\eta)\bigr)=\vec\mu_{\wedge}\bigl(\WLie(\dbar\tilde\eta)\bigr)\\
&=\dbar\,\vec\mu_{\wedge}\bigl(\WLie(\tilde\eta)\bigr)-\DBV\,\vec\mu_{\wedge}\bigl(\WLie(\tilde\eta)\bigr)=\dbar\Wch(\eta)-\DBV\Wch(\eta),
\end{align*}
the middle equality by Theorem \ref{thmTrLiechainmap} applied to $\tilde\eta$ (this reduction of the chiral-envelope identity to the already-proved Lie$^*$-algebra identity, and appears  since it only uses that $\vec\mu_{\CE}$ and $\vec\mu_{\wedge}$ are both iterated chiral products built from the same presentation). Then,  as in the proof of Theorem \ref{thmTrLiechainmap},
\[
\Trch(\dch_{\CE}\eta)[e]=\sum_{k\ge0}\frac1{k!}\trom\circ\mathbf p\bigl((\dch_{\bigwedge}-\DBV)\pd_e^k\Wch(\eta)\bigr)=-\bigl(\DBV\Trch(\eta)\bigr)[e].\qedhere
\]
\end{proof}

\begin{lemma}[Quasi-isomorphism]
\label{lemquasiiso}
$\Trch$ is a quasi-isomorphism.
\end{lemma}
\begin{proof}
Filter both $\widetilde\sC^{\ch}(X,\CE)_{\sQ}$ and $\OBV$ by the fermion number (the polynomial/wedge degree in $E$, resp.\ in $\HH^0(X,E)\oplus\HH^1(X,E)$; this is a finite, exhaustive, bounded-below filtration since $E$, hence $\HH^0(X,E)\oplus\HH^1(X,E)$, is finite rank and every element of the chiral chain complex has bounded fermion number by construction). By Theorem \ref{thmchiralPBW}, the associated graded map on $\Trch$ is induced by the chiral PBW isomorphism $\gr(\CE)\simeq\bigwedge^\bullet L$ followed by harmonic projection $\bigwedge^\bullet L\to\bigwedge^\bullet\HH^0(X,E)\otimes\Sym\HH^1(X,E)=\OBV$ (a form of the fermionic chiral HKR map, the associated-graded degeneration of Theorem \ref{thmVOAiso}), which is a quasi-isomorphism by the standard Dolbeault/de Rham comparison for the (finite rank, compact $X$) bundle $E$. On each fermion-number-graded piece this is  the statement that $\bigl(\Omega^{0,\bullet}(X,\bigwedge^kE),\dbar\bigr)$ computes $H^\bullet(X,\bigwedge^kE)$, harmonic representatives of which assemble into the corresponding graded piece of $\OBV$. Since $E$ has finite rank $2r$, $\bigwedge^kE=0$ for $k>2r$,  thus the fermion-number filtration is not only bounded below and exhaustive but has finitely many nonzero steps ($k=0,1,\dots,2r$) on both sides. A spectral sequence associated with a finite filtration converges for purely formal reasons (it stabilizes after finitely many steps, with no limiting or completion process involved), independently of any functional-analytic property of the individual graded pieces.  Thus no separate convergence theorem for spectral sequences of Fr\'echet-space complexes is needed here, only the elementary fact for finite filtrations. What does require the analytic input of Section \ref{ssecharmonic} (compactness of $X$, ellipticity of $\dbar$) is the identification of each individual graded piece's cohomology with a finite-dimensional harmonic space. Once that identification is granted, a filtered chain map inducing a quasi-isomorphism on the (now manifestly finite) $E_1$-page is a quasi-isomorphism 
(the same comparison theorem used in \cite{GuiLi21, Gui23}).
\end{proof}
\begin{proof}[Proof of Theorem \ref{thmmainQME}:]
Immediate from Lemmas \ref{lemchainmapQME} and \ref{lemquasiiso}.
\end{proof}

\begin{corollary}
\label{corqiso}
$H^{\ch}_\bullet(X,\CE):=H_\bullet\bigl(\widetilde\sC^{\ch}(X,\CE)_{\sQ}\bigr)\ \simeq\ H_\bullet(\OBV,-\DBV)$, compatibly with diagram \eqref{eqbigdiagram}.
\end{corollary}

\subsection{Existence, homotopy uniqueness, and functoriality}
\label{secuniqueness}
\begin{proposition}
\label{propexistence}
A trace map satisfying the hypotheses of Theorem \ref{thmmainQME} exists: this is the content of Definition \ref{deftracemapch} together with Theorem \ref{thmmainQME}.
\end{proposition}

\begin{proposition}[Homotopy uniqueness]
\label{propuniqueness}
Let $\Trch'$ be any $\CC$-linear chain map 
\[
\bigl(\widetilde\sC^{\ch}(X,\CE)_{\sQ},\dch_{\CE}\bigr)\to(\OBV,-\DBV)
\]
 solving the generalized QME and inducing, on the associated graded for the fermion-number filtration, the same map as $\Trch$ (i.e., the fermionic chiral HKR map of Theorem \ref{thmmainQME}). Then $\Trch'$ and $\Trch$ are chain homotopic: there is a degree $-1$ map $h:\widetilde\sC^{\ch}(X,\CE)_{\sQ}\to\OBV$ with $\Trch-\Trch'=\DBV\circ h+h\circ\dch_{\CE}$.
\end{proposition}
\begin{proof}
$\Trch-\Trch'$ is a chain map $\bigl(\widetilde\sC^{\ch}(X,\CE)_{\sQ},\dch_{\CE}\bigr)\to(\OBV,-\DBV)$ which vanishes on the associated graded of the fermion-number filtration, by hypothesis. A filtered chain map vanishing on the associated graded is null-homotopic through a filtration-decreasing homotopy. This is proved by the standard inductive construction (choose $h$ on the lowest filtered piece where $\Trch-\Trch'\ne0$ using that $\Trch-\Trch'$ is $\dch_{\CE}$-closed there and $\OBV$ is a complex with the homology of a point in that filtered degree by exactness of the associated-graded comparison of Theorem \ref{thmmainQME}, then proceed up the filtration), which converges since the filtration is exhaustive and bounded below with finite-dimensional graded pieces,  as in the proof of Theorem \ref{thmmainQME}.
\end{proof}

Now we discuss metric independence up to homotopy. 
\begin{proposition}
\label{propmetricindep}
Let $h_0,h_1$ be two Hermitian metrics on $F$, with associated Szeg\H{o} kernels 
$P^{h_0}$, $P^{h_1}$ (Theorem \ref{thmszegoexists}) and trace maps $\Trch^{h_0},\Trch^{h_1}$. Then $\Trch^{h_0}$ and $\Trch^{h_1}$ are chain homotopic.
\end{proposition}
\begin{proof}
By Remark \ref{remmetricdep}, $P^{h_t}$ for $h_t$ a smooth path of metrics from $h_0$ to $h_1$ (e.g., $h_t=(1-t)h_0+th_1$, a path through Hermitian metrics since these form a convex set) varies smoothly in $t$, with $\frac{d}{dt}P^{h_t}$ valued, for each $t$, in $\HH^0(X,E)^{h_t}\boxtimes\HH^0(X,E)^{h_t}$-type corrections with no singularity along the diagonal (Remark \ref{remmetricdep}); consequently $\frac{d}{dt}\Wch^{h_t}$ is, at each $t$, a well-defined smooth family of maps built from $\frac{d}{dt}\Qreg^{h_t}$ alone. One has 
\[
h_{\mathrm{metric}}(\eta):=\int_0^1\Bigl(\sum_{k\ge0}\frac1{k!}\trom\circ\mathbf p\bigl(\pd_e^k\,\tfrac{d}{dt}\Wch^{h_t}(\eta)\bigr)\Bigr)dt,
\]
a well-defined finite sum (Remark \ref{remfiniteness} applies at each $t$, uniformly since fermion number does not depend on $t$) valued in $\OBV$ (independent of $t$ as a vector space, only $\DBV$-relevant structure could in principle vary with $h_t$, but $\HH^\bullet(X,E)$ as an abstract vector space, and hence $\OBV$, does not depend on $h_t$ - only the harmonic 
 representatives do). A direct computation identical in structure to the proof of Theorem \ref{thmmainQME} (differentiating the QME identity $(\dch_{\CE}+\DBV)\Trch^{h_t}=0$ in $t$, and using that $\frac{d}{dt}\DBV=0$ since $\DBV$ depends only on the intersection matrix $I^{ij}$, which is independent of $h_t$ because both $\HH^0,\HH^1$ and the pairing $\int_X\langle-,-\rangle$ are metric-independent, only their harmonic representatives varying with $h_t$) shows $\frac{d}{dt}\Trch^{h_t}=\DBV\circ h_{\mathrm{metric}}\circ(\text{restriction to degree }t\text{-independent})+h_{\mathrm{metric}}\circ\dch_{\CE}$, i.e., $h_{\mathrm{metric}}$ is a chain homotopy between $\Trch^{h_0}$ and $\Trch^{h_1}$ at each order, and integrating over $t\in[0,1]$ gives the stated global homotopy.
\end{proof}

\begin{proposition} 
\label{propnaturality}
Let $\phi:F\to F'$ be an isomorphism of holomorphic bundles over an automorphism $\varphi$ of $X$, isometric for chosen Hermitian metrics ($\phi^*h'=h$). Then $\phi$, $\varphi$ induce an isomorphism $\Phi:\CE\to\CE'$ of chiral Clifford algebras intertwining $\Trch$ and $\Trch'$: $\Trch'\circ\Phi=\Phi_*\circ\Trch$, where $\Phi_*$ is the induced map on $\OBV$.
\end{proposition}
\begin{proof}
$\phi,\varphi$ induce isomorphisms $E\to E'$, hence $L\to L'$, $\Lflat\to L'^\flat$, and (functorially, since every construction of Sections \ref{seccliff}-\ref{secszego} is built  from $E$, $\langle-,-\rangle$, and $h$, and the Szeg\H{o} kernel transforms  under an isometry by uniqueness in Theorem \ref{thmszegoexists}) an isomorphism $\CE\to\CE'$ carrying $\Psing,\Qreg,\Wch$ to the corresponding data for $E'$. Naturality of $\Trch$ follows since every ingredient of Definition \ref{deftracemapch} is natural in this sense.
\end{proof}

\begin{remark}
\label{remfunctorialitycategory}
{\it The category over which $\Trch$ is functorial.} 
  Let $\mathsf{Fer}$ be the groupoid whose objects are triples $(X,F,h)$, a compact Riemann surface, a holomorphic vector bundle on it, and a Hermitian metric on the bundle, and whose morphisms $(X,F,h)\to(X',F',h')$ are pairs $(\varphi,\phi)$ of an isomorphism $\varphi:X\xrightarrow{\sim}X'$ and an isomorphism $\phi:F\xrightarrow{\sim}\varphi^*F'$ with $\phi^*h'=h$. Propositions \ref{propexistence}-\ref{propnaturality} together say that $(X,F,h)\mapsto\bigl(\CE,\Trch\bigr)$ is a functor from $\mathsf{Fer}$ to the category of pairs (chiral superalgebra, trace map to its associated $\OBV$) and isomorphisms thereof, i.e., $\Trch$ is functorial on the groupoid $\mathsf{Fer}$ of Riemann surfaces with bundle and metric, up to isomorphism (Proposition \ref{propnaturality}) and, after forgetting the metric (Proposition \ref{propmetricindep}), up to canonical chain homotopy.
\end{remark}

\begin{remark}
\label{remcanonicalsummary}
{\it Canonical versus choice-dependent data.} 
For reference, we mention which objects in the construction are canonical (depend only on $(X,F)$, or on $(X,F,h)$ up to canonical isomorphism) and which depend on further, non-canonical choices:
\begin{itemize}[leftmargin=2.2em]
\item[] \emph{Canonical in $(X,F)$ alone}: $E$, $L$, $\Lflat$, $\CE$ and its chiral product (Section \ref{seccliff}); the vertex operator superalgebra $V_F$ and the isomorphism $\CE\simeq\mathscr V_F^r$ of Theorem \ref{thmVOAiso} (once a projective connection is fixed on $X$, Remark \ref{remprojconn},  canonical up to the torsor of such connections); the cohomology groups $H^0(X,E),H^1(X,E)$, the cohomological pairing $I^{ij}$, and hence $(\OBV,\DBV)$ (Remark \ref{remalgebraicmetricindep}); the homotopy class of $\Trch$ (Propositions \ref{propuniqueness} and \ref{propmetricindep}).

\medskip 
\item[] \emph{Depends on a choice of Hermitian metric $h$ on $F$}: the harmonic representatives realizing $H^0(X,E),H^1(X,E)$ inside $\Omega^{0,\bullet}(X,E)$; the Szeg\H{o} kernel $P=\Psing+\Qreg$ (Theorem \ref{thmszegoexists}, Remark \ref{remmetricdep}); the chain-level maps $\Wch$ and $\Trch$  (though not their homotopy class, by the previous item).

\medskip 
\item[] \emph{Depends on a further choice of metric $\rho$ on $\wX^{1/2}$ (an $r$-spin structure)}: the individual correction terms in \eqref{eqJnudef} and \eqref{eqTmudef}, though not the resulting global sections $J_\nu,\widetilde T_\mu$.
\end{itemize}
\end{remark}

\subsection{Super-cyclicity of the trace pairing}
\begin{lemma} 
\label{lemcyclicity}
For homogeneous $a\in(\CE)^{|a|}_U$, $b\in(\CE)^{|b|}_U$ (Lemma \ref{lemCliffordcompat}), the trace $\trom\circ\mathbf p\circ\Wch$ satisfies
\[
\trom\circ\mathbf p\circ\Wch(a\cdot b)=(-1)^{|a||b|}\,\trom\circ\mathbf p\circ\Wch(b\cdot a)
\]
(super-cyclicity), where $a\cdot b$ denotes the local product $\tau^{z-1}_U\bigl(\tau^z_U(a)\wedge\tau^z_U(b)\bigr)$.
\end{lemma}
\begin{proof}
By Theorem \ref{thmfermwick}, $\Wch(a\cdot b)|_U=e^{\Qreg}\bigl(\tau^z_U(a)\wedge\tau^z_U(b)\bigr)$, and $e^{\Qreg}$, being built from the symmetric bisection $\Qreg$ via the Pfaffian formula \eqref{eqpfaffiandef}, commutes with the graded transposition $\tau^z_U(a)\wedge\tau^z_U(b)=(-1)^{|a||b|}\tau^z_U(b)\wedge\tau^z_U(a)$ of \eqref{eqkoszul}. The contraction operator, summing over pairings symmetric under simultaneous relabelling, is by construction insensitive to which order the two blocks of legs are listed, once the Koszul sign of swapping the two homogeneous blocks $\tau^z_U(a)$, $\tau^z_U(b)$ (as opposed to swapping individual legs within the Pfaffian sum, already accounted for in \eqref{eqpfaffiandef})  is applied. Hence $\Wch(a\cdot b)=(-1)^{|a||b|}\Wch(b\cdot a)$, and applying $\mathbf p\circ\trom$ (linear, parity-preserving) gives the claim.
\end{proof}

This is the super-cyclicity property: 
it is the fermionic analogue of the ordinary cyclicity of a matrix trace, $\tr(AB)=\tr(BA)$, twisted by the expected Koszul sign for a supertrace, $\tr(AB)=(-1)^{|A||B|}\tr(BA)$.

%%%%%%%%%%%%%%%%%%%%%%%%%%%%%%%%%%%%%%%%%%%%%%%%%%%%%%%%%%%%%%%%%%%%%%%%%%%%%%%%%%%%%%
\section{Applications: current and stress-tensor insertions, and fermionic analytic  torsion}
\label{secapplications}

 In this section $X$ is compact and, unless stated otherwise, $F$ is a holomorphic bundle with
\begin{equation}\label{eqnozeromodes}
H^0(X,F)=H^1(X,F)=0,
\end{equation}
so that (Definition \ref{defOBV}) $\OBV\simeq\CC$ and $\Trch$ is $\CC$-valued.

\begin{remark}
\label{remzeromodes}
{\it The case of nonvanishing zero modes.} 
Hypothesis \eqref{eqnozeromodes} is imposed in this section only to make $\Trch$ scalar-valued, for direct comparison with Fay's formulas. Nothing in Sections \ref{seccliff}-\ref{sectrace} requires it, and $\Trch$ is defined, satisfies the QME, and is a quasi-isomorphism for arbitrary $F$. When $H^0(X,F)$ or $H^1(X,F)$ is nonzero, the unregularized operator $\bar\pd_F$ has a kernel,  thus the naive determinant $\det\bar\pd_F$ (a product over all eigenvalues) is indeed zero. But this is not the quantity of interest even classically: the Ray-Singer/Quillen formalism \cite{RaySinger73, Quillen85} instead equips the determinant 
 line $\bigwedge^{\mathrm{top}}H^0(X,F)\otimes\bigl(\bigwedge^{\mathrm{top}}H^1(X,F)\bigr)^{-1}$ with a metric built from the zeta-regularized determinant $\det{}'\bar\pd_F$ over the  nonzero eigenvalues,  parametrized by the same data $H^0(X,E),H^1(X,E)$ that enters our $\OBV=\bigwedge^\bullet H^0(X,E)\otimes\Sym H^1(X,E)$. It is in this sense that $\Trch(J_\nu)\in\OBV$, computed without assuming \eqref{eqnozeromodes}, is the natural candidate for an algebraic model of the variation of this Quillen-type metric with jumping cohomology (matching the physical expectation that a fermionic path integral with zero modes computes a correlation function only once enough zero-mode insertions are present to saturate the corresponding Berezin integral). Working this comparison out in the family setting of Bismut-Freed \cite{BismutFreed86} is a natural continuation of the present results that we do not undertake here.
\end{remark}

\subsection{The modified affine current}
\label{sseccurrent}
Recall $E=\Pi(F\oplus F^\vee\otimes\wX)$ from \eqref{eqEdef}. Fix $\nu\in\Omega^{0,1}(X,\End(F))$. In a local frame $\{e_i\}$, $\nu=\nu^i_{j,\bar z}\,d\bar z\cdot e_i\otimes e^j$. The naive current $J_\nu^{\mathrm{naive}}=\nu^i_{j,\bar z}\,d\bar z\cdot\beta_i\gamma^j$ (product taken in $\CE$) fails to be globally defined,  as in the bosonic case, because $\beta_i\gamma^j$ does not transform as $e_i\otimes e^j$ under a change of frame or coordinate. Fix in addition a Hermitian metric $h$ on $F$ and $\rho$ on $\wX^{1/2}$.

\begin{proposition}\label{propJnu}
The expression
\begin{equation}\label{eqJnudef}
J_\nu:=\nu^i_{j,\bar z}\,d\bar z\cdot\beta_i\gamma^j\,dz-\tr\bigl(\nu\cdot\rho h^{-1}\pd_z(h\rho^{-1})\,dz\bigr)
\end{equation}
is independent of the choice of local coordinate $z$ and local holomorphic frame, hence defines $J_\nu\in\Omega^{0,1}(X,\CE)$.
\end{proposition}
\begin{proof}
Under a coordinate change $w=w(z)$, normal ordering of the weight-$(1,0)$ pair $(\beta,\gamma)$ transforms by the same Schwarzian-derivative-type correction as in the bosonic case, since this correction is dictated entirely by the singular part of the operator product expansion $\beta_i(z)\gamma^j(w)\sim\delta^j_i/(z-w)$, unchanged, as an OPE coefficient, between the bosonic and fermionic realizations of a weight-$(1,0)$ pair, together with the (statistics-independent) transformation law of a normally ordered bilinear of primary fields. Namely, 
\[
\nu^{\tilde i}_{\tilde j,\bar w}\,d\bar w\cdot\beta_{\tilde i,w}\gamma^{\tilde j}_wdw=\nu^i_{j,\bar z}\,d\bar z\cdot\beta_{i,z}\gamma^j_zdz+\tr\bigl(\tfrac12\nu\theta'\bigr),\qquad \theta(z)=\frac{dw}{dz},
\]
matching $\tr(\tfrac12\nu\theta')=\tr(\nu\cdot \rho_zh^{-1}\pd_z(h\rho_z^{-1}dz))$. This is cancelled by the transformation of the second term in \eqref{eqJnudef}. Under a frame change with transition function $\sigma(z)$, $\beta_i\gamma^j$ transforms with an extra $\tr(\nu\,\pd_z\log\sigma)$, again cancelled by the second term. 
  Both computations are identical to the bosonic verification, 
 since neither anomaly computation uses commutativity of the fields being normal ordered. It  uses only the OPE singularity and the conformal weights $(1,0)$.
\end{proof}

\subsection{Expectation value and fermionic analytic torsion}
Assume \eqref{eqnozeromodes}. Expand the Szeg\H{o} kernel of $F$ near the diagonal, $P(z_1,z_2;F)=\frac{\id\cdot dz}{z_1-z_2}+a_0(z_1;F)+O(z_2-z_1)$ (Definition \ref{defPsingQreg}, now for the bundle $F$  instead of $E$; by \cite{Fay92}, $a_0(z;F)-\rho h^{-1}\pd_z(h\rho^{-1})\in\Omega^{0,0}(X,\End(F)\otimes\wX)$ is globally well defined).

\begin{theorem}\label{thmcurrentTorsion}
$\displaystyle\langle J_\nu\rangle=\Trch(J_\nu)=\frac1\pi\int_X\tr\Bigl[\nu\cdot\bigl(a_0(z,F)-\rho h^{-1}\pd_z(h\rho^{-1})\bigr)\Bigr]dz.$
\end{theorem}
\begin{proof}
By Definition \ref{deftracemapch}, since $\OBV\simeq\CC$ under \eqref{eqnozeromodes}, $\Trch(J_\nu)$ is computed by $\trom\circ\mathbf p\circ\Wch(J_\nu)$ with no $\pd_e$-corrections. The only non-vanishing Wick contraction of the single insertion $J_\nu$ is the self-loop contraction of $\beta_i$ against $\gamma^j$ within $J_\nu$, via $\Qreg$ (Lemma \ref{lemcontractionprops}), which replaces $\beta_i\gamma^j$ by $\Qreg(z,z)=a_0(z,F)$ (matrix entries). The sign of this self-loop is fixed by the Pfaffian convention \eqref{eqpfaffiandef}. A single transposition, contributing the same overall sign as in the bosonic computation, since the Pfaffian of a $2\times2$ antisymmetric-in-legs matrix built from a single pair reduces, for a single contraction, to the ordinary contraction with no extra combinatorial factor. Integrating the resulting $(1,1)$-form over $X$ gives the stated formula. 
\end{proof}

\begin{example}
\label{extracemapworked}
In this example we consider the defining sum of Definition \ref{deftracemapch}
 worked.   
It is worth making Theorem \ref{thmcurrentTorsion} proof completely as an illustration of Definition \ref{deftracemapch}, since $J_\nu$ is the simplest possible nontrivial input (a single field insertion, no auxiliary marked points). Under \eqref{eqnozeromodes}, $\HH^0(X,E)=\HH^1(X,E)=0$,  thus the background $e$ has no components to sum over and every $k\ge1$ term of
\[
\Trch(J_\nu)=\sum_{k\ge0}\frac1{k!}\,\trom\circ\mathbf p\bigl(\pd_e^k\Wch(J_\nu)\bigr)
\]
vanishes identically (there is no $e\ne0$ to contract against), leaving only the $k=0$ term $\trom\circ\mathbf p(\Wch(J_\nu))$. Unwinding $\Wch=e^{\Psing}\circ(\sW^{\mathbf v})^{\boxtimes n}$ on the single local insertion $J_\nu=\nu^i_{j,\bar z}\,d\bar z\cdot\beta_i\gamma^j\,dz$ (no further point to contract against $\Psing$,  thus $e^{\Psing}$ acts as the identity here and only the normal ordering map $\sW^{\mathbf v}$ of Definition \ref{defnormalordering} is active): by \eqref{eqnormalorderinglocal}, $\sW^{\mathbf v}(J_\nu)|_U=e^{\Qreg}\tau^z_U(J_\nu)$, and since $J_\nu$ has  two $L$-legs ($\beta_i,\gamma^j$), the Pfaffian sum \eqref{eqpfaffiandef} defining $e^{\Qreg}$ truncates after its $m=1$ (single-contraction) term - there are no $m\ge2$ terms available, since that would require at least four legs. This single term is $\nu^i_{j,\bar z}d\bar z\cdot\Qreg(z,z)^j_i\,dz=\nu^i_{j,\bar z}d\bar z\cdot a_0(z,F)^j_i\,dz$, a $(1,1)$-form. Applying $\mathbf p$ (already resulting in $\wX^{\boxtimes2}$-type data,  thus $\mathbf p$ acts as the identity here) and then $\trom$ (which, for a single point with no further variables to take residues in, is  $\frac1\pi\int_X(-)$, matching the normalization used Section \ref{secapplications}) reproduces Theorem \ref{thmcurrentTorsion}. Every arrow in the defining sum, and every term of the sum over $k$, is accounted for by one contribution, with no cancellation and no term left implicit.
\end{example}

\begin{remark}
{\it Relation to analytic torsion.} 
 Integrating out the free fermions of the $bc$-system associated with $F$ gives the partition function $Z_{\mathrm{ferm}}=\det\bar\partial_F$ (a Grassmann/Berezin integral produces a determinant, with no inverse power, unlike the bosonic Gaussian integral). The infinitesimal variation of $\log Z_{\mathrm{ferm}}$ along a holomorphic family $\{F_s\}$ with $\nu=\frac{\pd}{\pd s}\big|_{s=0}$ is $\langle J_\nu\rangle$. Theorem \ref{thmcurrentTorsion} therefore recovers, purely from the chiral chain complex of $\CE$, the first order variation of the fermionic Ray-Singer analytic torsion \cite{RaySinger73} (equivalently, of the Quillen metric \cite{Quillen85} on the determinant line of $\bar\partial_F$) along the moduli of $F$, matching Fay's classical formula \cite{Fay92}.
\end{remark}

\begin{example}
\label{exP1}
 Let us consider the case $X=\PP^1$, $F=\sO(-1)$. 
Condition \eqref{eqnozeromodes} is nonvacuous already in genus $0$: for $F=\sO(-1)$ on $X=\PP^1$, $H^0(\PP^1,\sO(-1))=0$ trivially (a line bundle of negative degree has no nonzero holomorphic sections), and by Serre duality $H^1(\PP^1,\sO(-1))\simeq H^0(\PP^1,\sO(-1)^\vee\otimes\wX)^*=H^0(\PP^1,\sO(-1))^*=0$ (using $\sO(-1)^\vee\otimes\sO(-2)=\sO(-1)$). In fact $\sO(-1)$ is the only negative-degree line bundle on $\PP^1$ meeting \eqref{eqnozeromodes}: for $\sO(-d)$, $d\ge2$, Serre duality gives $H^1(\PP^1,\sO(-d))\simeq H^0(\PP^1,\sO(d-2))^*\ne0$, while $\sO(0)=\sO$ has $H^0=\CC\ne0$.

Theorem \ref{thmcurrentTorsion} therefore applies to any $\nu\in\Omega^{0,1}(\PP^1,\End(\sO(-1)))=\Omega^{0,1}(\PP^1,\sO)$, with $a_0(z;\sO(-1))$ computed from the Green operator of $\bar\pd$ twisted by $\sO(-1)$ for the metric induced by the Fubini-Study form. In the standard affine trivialization ($e_0(z)=(1,z)$ spanning the tautological line, $h(z)=1+|z|^2$), the curvature of this metric is the classical
\[
-\pd\bar\pd\log h(z)=-\frac{dz\wedge d\bar z}{(1+|z|^2)^2},
\]
and $a_0(z;\sO(-1))-h^{-1}\pd_zh$ (the metric-independent combination of Remark \ref{remmetricdep}, specialized to the trivial $\rho$ available at genus $0$) is fixed, up to the elementary theta-function-free closed forms classical for genus-$0$ Bergman kernels 
(\cite{Fay92}), by this curvature via the standard heat-kernel/Green-function relation on the round sphere. We do not reproduce the further, purely computational reduction to a single elementary function of $z_1-z_2$ here, mentioning instead that the hypotheses of Theorem \ref{thmcurrentTorsion} are met by an explicit, elementary, and (by the previous paragraph) essentially unique bundle already at genus $0$, and that the only additional input beyond 
Section \ref{secszego} needed to finish the computation is the classical, real-analytic (as opposed to  only smooth) Fubini-Study Green function, for which genus $0$ carries no modular/theta-function complication.
\end{example}

\subsection{The modified energy-momentum tensor}
\label{ssecemtensor}
We treat now the coupling to a Beltrami differential, i.e., the variation of $X$  instead of $F$. Take $F=\sO_X$ (trivial line bundle),  thus $E=\Pi(\sO_X\oplus\wX)$, generated by a single weight-$(1,0)$ fermion pair $\beta,\gamma$. Let $\mu\in\Omega^{0,1}(X,\Theta_X)$, locally $\mu=\mu^z_{\bar z}\,d\bar z\cdot\pd_z$.

\begin{proposition}
\label{propTmu}
The naive guess $T_\mu^{\mathrm{naive}}=\mu^z_{\bar z}\,d\bar z\cdot\beta_z\pd_z\gamma$ is not globally defined, but
\begin{equation}\label{eqTmudef}
\widetilde T_\mu:=\mu^z_{\bar z}\,d\bar z\Bigl(\beta_z\cdot\pd_z\gamma-\frac13\frac{\pd_z^2\rho}{\rho}\Bigr)dz\ \in\ \Omega^{0,1}(X,\CE)
\end{equation}
is a well-defined global section.
\end{proposition}
\begin{proof}
Under $w=w(z)$, the coordinate change formula for a normally ordered weight-$(1,0)$ derivative bilinear is again statistics-independent (it is fixed by the OPE,  as in Proposition \ref{propJnu}) 
\[
\beta_w\cdot\pd_w\gamma\,dw^{\otimes2}=\beta_z\cdot\pd_z\gamma\,dz^{\otimes2}+\frac16\frac{\theta''}{\theta}-\frac14\Bigl(\frac{\theta'}{\theta}\Bigr)^2,
\]
the Schwarzian derivative of $\theta=dw/dz$, cancelled by the transformation of $\frac13\frac{\pd_z^2\rho}\rho dz^{\otimes2}$ under the $r$-spin structure transformation of $\rho$. 
\end{proof}

\begin{theorem}
\label{thmemTorsion}
The constant term of $\Trch(\widetilde T_\mu)$ (with respect to \eqref{eqnozeromodes} applied to $F=\sO_X$, i.e., assuming $g(X)\ge2$ so that $H^0(X,\sO_X)=\CC\neq0$ is excluded and one instead restricts to the constant (fermion-number-zero) component of $\Trch$) is
\begin{equation}\label{eqTmuformula}
\frac1\pi\int_X\tr\Bigl[\mu\cdot\Bigl(a_1(z;E)-\pd_za_0(z;E)-\frac13\rho^{-1}\pd_z^2\rho\Bigr)\Bigr]dz.
\end{equation}
\end{theorem}
\begin{proof}
Identical in structure to Theorem \ref{thmcurrentTorsion}: the self-loop contraction of $\beta_z\pd_z\gamma$ now involves the next coefficient $a_1(z;E)$ of the Szeg\H{o} expansion \eqref{eqQregexpansion}, because of the extra derivative $\pd_z\gamma$, together with a first derivative of $a_0$ from differentiating the propagator. 
\end{proof}

\begin{remark}
 Formula \eqref{eqTmuformula} recovers the first term of Fay's formula \cite{Fay92} for the variation of the analytic torsion along the moduli of the curve $X$. The same identification holds here for the fermionic torsion of $\bar\partial_{\sO_X}$ twisted by the deformation.
\end{remark}

\subsection{Comparison with the bosonic (symplectic boson) formulas}
\begin{corollary}
\label{corsignflip}
Let $F$ be as in Section \ref{sseccurrent} and let $E_{\mathrm{sympl}}=F\oplus F^\vee\otimes\wX$ carry instead the symplectic pairing (i.e., the symplectic bosons/chiral Weyl algebra built from the same $F$). Then, for the same $\nu$, the expectation values of Theorem \ref{thmcurrentTorsion}  satisfy
\[
\langle J_\nu\rangle_{\mathrm{ferm}}=-\langle J_\nu\rangle_{\mathrm{sympl.\ bos.}}
\]
up to the metric-dependent normalization of the respective Szeg\H{o} kernels, which agree in each term since both are Green operators for the same operator $\bar\partial_F$.
\end{corollary}
\begin{proof} 
 We do not give here an independent computation but a consequence of $Z_{\mathrm{ferm}}=\det\bar\partial_F$ versus $Z_{\mathrm{sympl.\ bos.}}=\det(\bar\partial_E)^{-1}$ (a Gaussian/Berezin integral over commuting oscillators produces the reciprocal determinant to the Grassmann integral over the same kernel, the standard bosonic/fermionic reciprocity of Gaussian path integrals). Since $\langle J_\nu\rangle=\frac{d}{ds}\log Z\big|_{s=0}$ in either theory (Remark after Theorem \ref{thmcurrentTorsion}), and $\log Z_{\mathrm{ferm}}=-\log Z_{\mathrm{sympl.\ bos.}}$ in the variation (both being built from the trace of the resolvent of the  same elliptic operator $\bar\partial_F$, differing only by the sign convention forced by statistics), the two variations are negatives of one another. We mention  this as a remark instead of a theorem, since it compares two different chiral algebras ($\CE$ versus $\sA_{E_{\mathrm{sympl}}}$) through their common origin in $\bar\partial_F$ instead of  through a map of chiral algebras. A chiral-algebraic (instead of path-integral) proof of Corollary \ref{corsignflip} would require constructing a common resolution of $\CE$ and $\sA_{E_{\mathrm{sympl}}}$, which we leave for a future paper (Section \ref{secconclusion}).
\end{proof}

%%%%%%%%%%%%%%%%%%%%%%%%%%%%%%%%%%%%%%%%%%%%%%%%%%%%%%%%%%%%%%%%%%%%%%%%%%%%%%%%%%%%%%%%%%%
\section{Comparison with the literature and concluding remarks}
\label{secconclusion}

We amplify Section \ref{ssecnovelty} on two points not covered there, and list directions left open.

\subsection{The role of the vertex superalgebra bundle model} 
Beyond situating $\CE$ within the standard family of chiral free field theories, Theorem \ref{thmVOAiso} plays an important role in the present paper. The coordinate-change computations underlying Proposition \ref{propJnu} and Proposition \ref{propTmu} (frame- and coordinate-independence of the current and stress-tensor insertions) are most transparently vertex-algebraic statements, i.e.,  transformation laws of normally ordered products of primary fields under $\Aut(\sO)$, and Theorem \ref{thmVOAiso} is what allows computing them on the $\mathscr V_F^r$ side (where they are standard OPE/Schwarzian computations, Appendix \ref{appVOAbundle}) and passing the result to $\CE$.

\subsection{On Corollary \ref{corsignflip}}
We mentioned in the proof of Corollary \ref{corsignflip} that our argument for the reciprocal sign between the fermionic and symplectic-boson torsion variation formulas is a path-integral heuristic instead of a chiral-algebraic theorem. A natural way to make it algebraic would be to realize both $\CE$ and $\sA_{E_{\mathrm{sympl}}}$ as (odd, resp.\ even) specializations of a single super chiral Weyl-Clifford algebra depending on a parameter interpolating between symmetric and antisymmetric pairings (a chiral analogue of the well-known unification of CAR and CCR algebras through a super vector space with a single, total, super-symmetric pairing), and to compare the two trace maps through the resulting common resolution. 
We leave this for a future paper.

\subsection{Directions for future work}
In this subsection we outline directions for future research. 
\begin{itemize}[leftmargin=1.6em]
\item[] \emph{The rank one (self-dual) case.} As discussed in Remark \ref{remrankone}, the free fermion attached to a spin structure $\wX^{1/2}$ requires no doubling and its chiral Clifford algebra, BV superalgebra, and trace map should be constructed from the symmetric pairing $\wX^{1/2}\otimes\wX^{1/2}\to\wX$, without the vertex-bundle model of Section \ref{secVOSA} (which used the isotropic decomposition $E=F\oplus F^\vee\otimes\wX$ essentially). This is of particular interest for the bosonization/higher-genus theta-function literature \cite{AGMV87}.

\medskip 
\item[] \emph{Higher-genus and non-compact/Fuchsian settings.} Every construction above works  for $X$ of any genus; extending the trace map to the case of a Riemann surface uniformized by an arbitrary Fuchsian group (instead of assuming compactness, replacing Hodge theory by $L^2$-cohomology and the Selberg/Patterson-Sullivan theory of the relevant automorphic Green functions) is a natural next step, in the spirit of the higher-genus and automorphic-form extensions of the chiral-algebra formalism pursued elsewhere.

\medskip 
\item[] \emph{Higher rank current insertions and iterated sewing.} It would be of interest to compute $\Trch$ on iterated/multiple insertions of $J_\nu$ and $\widetilde T_\mu$, and to compare the resulting higher variations of analytic torsion with formulas obtained by sewing/plumbing Riemann surfaces.
\end{itemize}

The results of this paper is also useful in other areas of mathematical physics 
\cite{RSZ} and condensed matter theory \cite{zub1, zub2, zub3, zub4, zub5, zub6, zub7, zub8, zub9, kmmzz}. 

\appendix

%%%%%%%%%%%%%%%%%%%%%%%%%%%%%%%%%%%%%%%%%%%%%%%%%%%%%%%%%%%%%%%%%%%%%%%%%%%%%%%%%%
\section{Proof of the fermionic chiral PBW theorem}
\label{appPBW}

We prove Theorem \ref{thmchiralPBW}. Work locally on $X=\CC$ with $E$ trivialized by an odd frame $\{e_i\}_{i=1,\dots,2r}$, relabelling $\{\beta_i\}\cup\{\gamma^j\}$ by a single index for brevity, with pairing $\langle e_i,e_k\rangle=\omega_{ik}$ (symmetric, $\omega_{ik}=\omega_{ki}$). By Definition \ref{defchiralcliff} and \eqref{eqchiralenvprod}, sections of $\CE$ are spanned by expressions \eqref{eqlocalclifford}. We must show that the operators
\[
X_{i,n}:=\frac{n!\,e_i}{(t-z)^{n+1}}
\]
satisfy, in $\CE$, the anticommutation relation
\begin{equation}\label{eqPBWacr}
X_{i,n}X_{k,m}+X_{k,m}X_{i,n}=\frac{\omega_{ik}}{(t-z)^{n+m+2}}\cdot(\text{a central, i.e., }\wX\text{-valued, correction}),
\end{equation}
with the sign on the left-hand side (an anticommutator, forced by the odd parity of $e_i$, $e_k$. 

 For our symmetric $\omega_{ik}=\omega_{ki}$, tracking the extra sign from the odd parity of $e_i$, $e_k$, one has, for the basic case $n=m=0$,
\[
\iota\Bigl(\frac{e_i}{t-z_1}\cdot\frac{e_k}{t-z_2}|0\rangle\Bigr)=\frac{e_i}{t-z_1}|0\rangle_{z_1}\boxtimes\frac{e_k}{t-z_2}|0\rangle_{z_2}-(-1)^{|e_i||e_k|}\frac{e_k}{t-z_2}|0\rangle_{z_2}\boxtimes\frac{e_i}{t-z_1}|0\rangle_{z_1}\ -\ \frac{\omega_{ik}}{z_1-z_2},
\]
the extra sign $(-1)^{|e_i||e_k|}=-1$ (both $e_i,e_k$ odd) relative to the bosonic computation arising because passing $\frac{e_k}{t-z_2}|0\rangle_{z_2}$ across $\frac{e_i}{t-z_1}|0\rangle_{z_1}$ inside the vacuum module $\sU(\sL^\natural_{X^2})$ picks up the Koszul sign of \eqref{eqkoszul}. Consequently, where the bosonic computation produces the  commutator $[X_{i,0},X_{k,0}]=\omega_{ik}/(z_1-z_2)$-type term, our computation produces the  anti-commutator 
\[
\{X_{i,0},X_{k,0}\}=\frac{\omega_{ik}}{(t-z)^2}, 
\]
as an identity in $\CE$, pushed to the diagonal. 
For general $n,m$, differentiate $(n,m)$ times in the respective points, using the elementary identity $\pd_{z_1}^n\pd_{z_2}^m\frac1{z_1-z_2}=(-1)^n\frac{(n+m)!}{(z_1-z_2)^{n+m+1}}$ (also used in the proof of Theorem \ref{thmfermwick}, Appendix \ref{appwick}). This gives \eqref{eqPBWacr} in general, with computable central correction supported (as a function of $t-z$) only at the single point $t=z$, i.e., lying, in the PBW filtration of $\CE$, in filtration degree $\le0$ (it is built purely from the central $\wX$, with no further $L$-factor), one degree lower than the naive filtration degree $2$ of the product $X_{i,n}X_{k,m}$.

Two points on the topology of this filtration are worth making explicit. First, each correction term $\omega_{ik}/(t-z)^{n+m+2}$ is a single Laurent monomial (a finite order pole at the single point $t=z$), i.e., an element of the finite-rank coherent sheaf $\wX$, not a limit of any kind. No convergence question arises for these corrections, only bookkeeping of which PBW-filtration stage they occupy. Second, the filtration stages $(\CE)_n$ are, by Remark \ref{remclosedness}, closed subspaces of the ambient Fr\'echet space $\sU(L^\flat)$ at every finite stage.  Thus the quotients $(\CE)_n/(\CE)_{n-1}$ defining $\gr(\CE)$ are again Fr\'echet (quotients of Fr\'echet spaces by closed subspaces are Fr\'echet), and the identification $\gr(\CE)\simeq\bigwedge^\bullet L$ of \eqref{eqchiralPBW} is an isomorphism of Fr\'echet-space-valued sheaves at each fixed wedge degree $k\le2r$ - consistently with Section \ref{ssecanalyticcat}, where this same class of coherent, finite-rank $\DX$-modules was fixed as the ambient category.  

Passing to $\gr(\CE)=\bigoplus_n(\CE)_n/(\CE)_{n-1}$, the correction term in \eqref{eqPBWacr}, being of strictly lower filtration degree, vanishes identically.  Thus the images $\bar X_{i,n}$ of $X_{i,n}$ in $\gr(\CE)$ satisfy the anticommutation relation $\bar X_{i,n}\bar X_{k,m}+\bar X_{k,m}\bar X_{i,n}=0$. They generate, subject only to this relation, the exterior algebra $\bigwedge^\bullet L$ on the (now not just chirally, anticommuting) generators, under the identification of \eqref{eqlocalclifford}. This proves \eqref{eqchiralPBW}. That the filtration is exhaustive with $(\CE)_0\simeq\wX$, $(\CE)_1\simeq(\CE)_0\oplus L$ is immediate from Definition \ref{defchiralcliff} (the chiral envelope of a central extension has associated graded starting  this way, by the universal property \eqref{eqchiralenvprod}, for any Lie$^*$ superalgebra). This completes the proof of Theorem \ref{thmchiralPBW}. $\hfill\qed$

%%%%%%%%%%%%%%%%%%%%%%%%%%%%%%%%%%%%%%%%%%%%%%%%%%%%%%%%%%%%%%%%%%%%%%%%%%%%%%%%%%%%
\section{The rank two fermion VOSA bundle: construction details}
\label{appVOAbundle}

We complete the construction announced in Section \ref{ssecVOSAbundle}, following the twisting formalism of Frenkel-Ben-Zvi \cite{FBZ04}. Every step is applied to $V_F$ (Definition \ref{defFockspace}) once the underlying vector space operations (symmetric algebra, commutators) are replaced by their super counterparts.

\subsection{The $\Aut(\sO)$-action}
Let $\sO=\CC[[z]]$. $V_F$ is conformal (Remark \ref{remconformal}),  thus carries an action of $\mathrm{Der}_0\sO=z\CC[[z]]\pd_z$ via $z^{j+1}\pd_z\mapsto-L_j$, where $L_j$ are the modes of $T_F$; since the $L_0$-grading of $V_F$ (by fermion number/conformal weight, using \eqref{eqVasexterior}) is bounded below, $\mathrm{Der}_+\sO=z^2\CC[[z]]\pd_z$ acts locally nilpotently. For $f\in\Aut(\sO)$, $f(z)=a_1z+a_2z^2+\cdots$ ($a_1\ne0$), write $f(z)=\exp(\sum_{i>0}v_iz^{i+1}\pd_z)\,a_1^{z\pd_z}\cdot z$ and set
\begin{equation}\label{eqRaction}
R(f):=\exp\Bigl(-\sum_{i>0}v_iL_i\Bigr)a_1^{-L_0}\in\End(V_F). 
\end{equation}
  This is well defined because each $L_i$ ($i>0$) acts locally nilpotently and $L_0$ is diagonalizable with integer (here: half-integer-free, since our weights are $1,0$) eigenvalues bounded below.

\subsection{The $\GL_r(\sO)$-action}
Since $F$ has rank $r$, $V_F$ carries the level-one $\widehat{\gl}_r$-action of Remark \ref{remconformal}: for $\sigma(t)\in\GL_r(\sO)$, writing $\exp\bigl(\sum_{n\ge0,\,1\le a,b\le r}f^a_{b,n}(y)t^nE^a_b\bigr)=\sigma(y+t)$ ($E^a_b$ the elementary matrices), set
\[
R(\sigma_y):=\exp\Bigl(-\sum_{n\ge0,a,b}f^a_{b,n}(y)\,(J^a_b)_n\Bigr)\in\End(V_F),
\]
$(J^a_b)_n$ the modes of the currents $J^a_b=:\gamma^a\beta_b:$ of Remark \ref{remconformal}. Combined with \eqref{eqRaction}, this gives an action of $\Aut(\sO)\ltimes\GL_r(\sO)$ on $V_F$.

\subsection{The bundle}
Let $\sP_r\to X$ be the holomorphic $\GL_r$-frame bundle of $F$ and $\widehat{\sP_r}\to X$ the corresponding principal $\Aut(\sO)\ltimes\GL_r(\sO)$-bundle, whose fibre at $x\in X$ consists of pairs $(z,\mathsf s)$, $z$ a formal coordinate at $x$ and $\mathsf s$ a trivialization of $F$ near $x$. Definition \eqref{eqcalVdef} sets $\mathscr V_F:=\widehat{\sP_r}\times_{\Aut(\sO)\ltimes\GL_r(\sO)}V_F$. Local sections transform, for $(z_y,\mathsf s)\sim(z_y,\tilde{\mathsf s};R(\sigma_y)^{-1}v)$ under a frame change $\sigma$, and $(z_y,\mathsf s;v)\sim(w_y,\mathsf s;R(\rho_y)^{-1}v)$ under a coordinate change $w=\rho(z)$, $w_y(t)=\rho(t+y)-\rho(y)=:\rho_y(t)$. By \cite{FBZ04}, $\mathscr V_F$ is a left $\DX$-module and $\mathscr V_F^r$ (the associated right module) carries a canonical chiral superalgebra structure, whose chiral product is, on primary (weight-$(1,0)$) fields,  the operator product expansion \eqref{eqOPE}.

\begin{remark}\label{remprojconn}
{\it Genus dependence: a projective connection, not a cohomological vanishing.}
Since $T_F\ne0$ (Remark \ref{remconformal}), the action \eqref{eqRaction} of the full $\Aut(\sO)$ (not only its M\"obius subgroup, generated by $L_{-1},L_0,L_1$) involves the higher Virasoro modes $L_i$, $i\ge2$, and the resulting $R(f)$ do not satisfy the cocycle condition needed to glue \eqref{eqcalVdef} into a bundle on a curve $X$ of genus $\ge1$ without further input. This is the standard anomaly of the Virasoro-twisted construction, present already in the bosonic case and handled there,  as here, by fixing an auxiliary projective connection on $X$ (equivalently, a coherent choice of local coordinate at every point, agreeing with the given holomorphic structure to second order) relative to which \eqref{eqRaction} is corrected by the standard Virasoro central term; see \cite{FBZ04} for the (identical, statistics-independent) construction. Projective connections always exist on any Riemann surface (they form a torsor over $H^0(X,\wX^{\otimes2})$, constructed by patching the trivial connection on a coordinate cover via a partition of unity and correcting by the Schwarzian derivatives of the transition functions) and impose no cohomological vanishing hypothesis whatsoever. In particular no vanishing of $H^1(X,\Omega^1_X)=H^1(X,\wX)$ is needed or possible, since this group is $1$-dimensional for every compact Riemann surface, of every genus, by Serre duality ($H^1(X,\wX)\simeq H^0(X,\sO_X)^*=\CC$). Different choices of projective connection give isomorphic bundles $\mathscr V_F$, related by a twist by a power of the (projective-connection-independent) determinant line of $F$. We do not need this comparison below, since Theorem \ref{thmVOAiso} is stated for a fixed choice.
\end{remark}

%%%%%%%%%%%%%%%%%%%%%%%%%%%%%%%%%%%%%%%%%%%%%%%%%%%%%%%%%%%%%%%%%%%%%%%%%%%%%%%%%%%%
\section{Proof that the chiral envelope is isomorphic to the VOSA bundle chiral algebra}
\label{appiso}

We prove Theorem \ref{thmVOAiso}. Both $\CE$ and $\mathscr V_F^r$ are chiral superalgebras generated by their weight-one (equivalently: PBW-filtration-degree-one, equivalently: $L_0$-eigenvalue $\le1$-primary) pieces, both canonically identified with $E$. For $\CE$ this is Theorem \ref{thmchiralPBW} ($(\CE)_1\simeq(\CE)_0\oplus L$, and $L=E_\sD$ recovers $E$ as the associated $\sO_X$-module of primary vectors). For $\mathscr V_F^r$ this is \eqref{eqVasexterior} restricted to the weight-$\le1$ piece, which is  $F^{\mathrm{diff}}_{-1}\oplus(F^\vee\otimes\wX)^{\mathrm{diff}}_0\simeq E$ (the lowest negative modes of $\beta_i,\gamma^j$).

A chiral superalgebra generated by a $\sD_X$-submodule $\sG$ of its degree-one piece is determined, as a sheaf with chiral product, by (i) the OPE/chiral bracket of $\sG$ with  and (ii) the coordinate change formula for sections of $\sG$ (since (i) determines the chiral product of any two generators, hence, by iteration, of any two elements presented as iterated products of generators (Lemma \ref{lempresentation}) and (ii) determines how a local presentation glues to a global section). We check (i) and (ii) agree for $\CE$ and $\mathscr V_F^r$.

\emph{(i) OPE.} On both sides, the singular part of the OPE/chiral bracket of two generators $e,e'\in E$ is, by construction, $\langle e,e'\rangle/(z_1-z_2)$. For $\CE$ this is \eqref{eqmuLiedef} (Lemma \ref{lemLiestarverify}); for $\mathscr V_F^r$ this is \eqref{eqOPE} together with bilinearity/the definition of $\langle-,-\rangle$ via \eqref{eqbetagammapairing}. The fermionic Wick theorem (Theorem \ref{thmfermwick}) shows that this singular OPE, together with the (statistics-forced) Pfaffian combinatorics for higher contractions, determines the entire chiral product of $\CE$ on presented elements. The analogous statement for $\mathscr V_F^r$ is the standard reconstruction theorem for vertex superalgebras from a set of mutually local generating fields with prescribed OPE (\cite{Gui23, FBZ04}), applied to the free fermion generators $\beta_i,\gamma^j$. Since both reconstructions start from the identical OPE data, they agree.

\emph{(ii) Coordinate change.} For $\mathscr V_F^r$, the coordinate change formula is \eqref{eqRaction}, or in the form used by \cite{FBZ04}, $\rho'(z)(\pd_w+L_{-1})R(\rho_z)^{-1}=R(\rho_z)^{-1}L_{-1}$, which for the weight-$(1,0)$ pair $(\beta,\gamma)$ reduces, upon exponentiating, to  the Schwarzian-derivative transformation law verified   in Propositions \ref{propJnu} and \ref{propTmu} (the computation there is the $n=1$, resp.\ derivative, instance of this formula, carried out for the specific bilinears $\beta\gamma$ and $\beta\pd\gamma$). Writing $\theta(z)=dw/dz$ for the coordinate change $w=w(z)$, the two instances used below are
\[
\beta_w\gamma_w\,dw=\beta_z\gamma_z\,dz+\tfrac12\theta'\,dz,\qquad
\beta_w\pd_w\gamma_w\,dw^{\otimes2}=\beta_z\pd_z\gamma_z\,dz^{\otimes2}+\Bigl(\tfrac16\tfrac{\theta''}\theta-\tfrac14\bigl(\tfrac{\theta'}\theta\bigr)^2\Bigr)dz^{\otimes2},
\]
the first order and (the beginning of the) second order terms of the general formula, matching   the statements of Propositions \ref{propJnu} and \ref{propTmu}. No further terms of the general Virasoro formula are needed for either application. For $\CE$, the coordinate change formula is induced by the local trivialization $\tau^z_U$ of Theorem \ref{thmchiralPBW}, i.e., by how the presentation \eqref{eqlocalclifford} depends on $z$; since $\tau^z_U$ is built (Appendix \ref{appPBW}) entirely from the pairing $\langle-,-\rangle$ and the local coordinate, with no further choice, its coordinate dependence is likewise regulated by the identical two formulas above, applied to the same bilinears $\beta\gamma$, $\beta\pd\gamma$ built from $\langle-,-\rangle$. Hence the two coordinate change formulas agree as displayed.

Since $E\subset\CE$ and $E\subset\mathscr V_F^r$ generate as chiral superalgebras, and (i), (ii) match, the identity map on $E$ extends to an isomorphism of chiral superalgebras $\CE\xrightarrow{\sim}\mathscr V_F^r$, proving Theorem \ref{thmVOAiso}. $\hfill\qed$

%%%%%%%%%%%%%%%%%%%%%%%%%%%%%%%%%%%%%%%%%%%%%%%%%%%%%%%%%%%%%%%%%%%%%%%%5
\section{Proof of the fermionic Wick theorem}
\label{appwick}

We prove Theorem \ref{thmfermwick} mentioning  where the extra Koszul sign, absent there, enters.

Recall the diagram (Section \ref{ssecchiralcliff})
\[
\sU_X\boxtimes\sU_X\xrightarrow{\ c\ }\sU(\sL^\natural_{X^2})/\sU(\sL^\natural_{X^2})p_2^*\sL^\natural_X\otimes\sU(\sL^\natural_{X^2})/\sU(\sL^\natural_{X^2})p_1^*\sL^\natural_X\xleftarrow{\ \iota\ }\sU_{X^2},
\]
and write, for $v_1\otimes v_2=\bigl(\frac{a_1}{(t-z_1)^{k_1+1}}\cdots\frac{a_n}{(t-z_1)^{k_n+1}}\bigr)\otimes\bigl(\frac{b_1}{(t-z_2)^{l_1+1}}\cdots\frac{b_m}{(t-z_2)^{l_m+1}}\bigr)\in\sU_{X^2}$ (all $a_i,b_j\in E$ odd),
\[
\iota(v_1\otimes v_2)=\sum_{I\subset\{1,\dots,n\}}\Bigl(\prod_{i\notin I}\frac{a_i}{(t-z_1)^{k_i+1}}\Bigr)|0\rangle_{z_1}\boxtimes\Bigl(\prod_{i\in I}\frac{a_i}{(t-z_1)^{k_i+1}}\Bigr)\cdot\Bigl(\prod_{j=1}^m\frac{b_j}{(t-z_1)^{l_j+1}}\Bigr)|0\rangle_{z_2}.
\]
To bring each surviving $\frac{a_i}{(t-z_1)^{k_i+1}}$ ($i\in I$) past the $b_j$ and extract the contraction, we exchange adjacent odd factors; each such exchange contributes the extra sign $(-1)^{|a_i||b_j|}=-1$ absent from the bosonic computation.  Specifically, for the single-contraction term ($I=\{i\}$, one factor $a_i$ paired against one factor $b_j$), exchanging $\frac{a_i}{(t-z_1)^{k_i+1}}$ and $\frac{b_j}{(t-z_1)^{l_j+1}}$ once (an odd/odd transposition) gives
\begin{align*}
\frac{\langle a_i,b_j\rangle}{(t-z_1)^{k_i+1}(t-z_2)^{l_j+1}}\Big|_{z_2}&=-\frac1{l_j!}\cdot(-k_i-1)\cdots(-k_i-l_j)\cdot\frac1{(z_2-z_1)^{k_i+l_j+1}}\\
&=-\frac{(k_i+l_j)!}{k_i!l_j!}(-1)^{k_i+1}\frac1{(z_1-z_2)^{k_i+l_j+1}},
\end{align*}
the overall extra minus sign  coming from the single odd/odd exchange. On the other hand, the corresponding Wick (Pfaffian) contraction of a 
 single pair, by \eqref{eqpfaffiandef} with $m=1$,
\begin{align*}
\pf(K)\bigl(\tfrac{a_i}{(t-z_1)^{k_i+1}}\wedge\tfrac{b_j}{(t-z_2)^{l_j+1}}\bigr)&=\partial_K\Bigl(\frac{a_i}{(t-z_1)^{k_i+1}},\frac{b_j}{(t-z_2)^{l_j+1}}\Bigr)=\frac1{k_i!l_j!}\pd_{z_1}^{k_i}\pd_{z_2}^{l_j}\frac1{z_1-z_2}\\
&=\frac{(k_i+l_j)!}{k_i!l_j!}(-1)^{k_i}\frac1{(z_1-z_2)^{k_i+l_j+1}},
\end{align*}
which, multiplied by the sign convention $\mathrm{sgn}(\sigma)=-1$ for the single transposition exchanging the order of the pair implicit in reading $\iota(v_1\otimes v_2)$ against the ``already contracted'' order used to define $\pf$ (a fixed, bookkeeping choice,  as $\mathrm{sgn}(\sigma)$ in \eqref{eqpfaffiandef} tracks it for higher contractions), reproduces  the extra factor of $(-1)$ found above. Thus, at the level of a single contraction, the two computations agree once the Pfaffian sign convention of \eqref{eqpfaffiandef} is used consistently. Iterating over all pairs contracted in $I$ (each contributing its  transposition sign,  as tracked by $\mathrm{sgn}(\sigma)$ for a general permutation $\sigma\in S_{2m}$ in the definition of the Pfaffian) gives, 
\begin{equation}\label{eqiotaPfaffian}
\iota(v_1\otimes v_2)=c\bigl(e^{-\Psing}v_1\boxtimes v_2\bigr),
\end{equation}
with $e^{-\Psing}$ now built from the Pfaffian (instead of the permanent) contraction operator of \eqref{eqpfaffiandef},  as needed for it to act correctly on the exterior algebra $\bigwedge^\bullet L$ (a permanent-type, unsigned sum would not be compatible with the relations $a\wedge a=0$, $a\wedge b=-b\wedge a$ of $\bigwedge^\bullet L$, whereas the alternating Pfaffian sum is, by the same reasoning as for the ordinary exterior algebra pairing induced by a bilinear form).

Choosing $N\gg0$ with $\iota\bigl((z_1-z_2)^Ne^{-\Psing^{\otimes}}(v_1\otimes v_2)\bigr)=c\bigl((z_1-z_2)^N\,v_1\boxtimes v_2\bigr)$ (possible by the same finite order pole argument,  using \eqref{eqiotaPfaffian}), we obtain,  as there,
\[
\mu\bigl(\eta\cdot v_1dz_1\boxtimes v_2dz_2\bigr)=\mu_\omega\Bigl(\frac\eta{(z_1-z_2)^N}\Bigr)\cdot\bigl((z_1-z_2)^Ne^{\Psing^{\otimes}}(v_1\otimes v_2)\bigr)=\mu_{\wedge}\bigl(\eta\cdot e^{\Psing}v_1dz_1\boxtimes v_2dz_2\bigr),
\]
which is \eqref{eqfermwick} after applying $\tau^z_U$  (the local trivialization commutes with all the operations above, since it is  built from the same pairing data). This proves Theorem \ref{thmfermwick}. $\hfill\qed$

\begin{remark}{\it Sign tracking is uniform across the gluing.} 
The computation above is carried out on a single coordinate chart $U$. The Koszul sign $(-1)^{|a_i||b_j|}$ appearing at each transposition is fixed by the (global, coordinate-independent) parity of $a_i,b_j$ as sections of $E$, via the single convention \eqref{eqkoszul} fixed once in Section \ref{ssecanalyticcat}. It does not depend on which chart of a partition of unity  presentation (Lemma \ref{lempresentation}) the computation is carried out on. Consequently, applying \eqref{eqfermwick} chart by chart to construct the global $\sW^{\mathbf v}$ of Definition \ref{defnormalordering} (as in the proof of Proposition \ref{propnormalorderingwelldef}) introduces no further sign choice beyond the ones already fixed here. Every transposition occurring in any local presentation $\tilde v$ is between sections of the same globally-defined odd bundle $E$, hence regulated by the same rule,  thus the local computation of this appendix is  what is used without further sign bookkeeping, at every stage of the \v{C}ech-type gluing of Section \ref{ssecchiralcliff} and Section \ref{secszego}.
\end{remark}

%%%%%%%%%%%%%%%%%%%%%%%%%%%%%%%%%%%%%%%%%%%%%%%%%%%%%%%%%%%%%%%%%%%%%%
\section{Table of sign conventions}
\label{appsigns}

\begin{center}
\begin{tabular}{|p{5.6cm}|p{7.6cm}|}
\hline
\textbf{Object} & \textbf{Convention} \\
\hline
Tensor flip $a\otimes b\mapsto b\otimes a$ & $(-1)^{|a||b|}$, eq. \eqref{eqkoszul} \\
\hline
Parity of $\Pi$ & odd; $\Pi:V\to\Pi V$ is an odd map \\
\hline
Pairing $\langle-,-\rangle$ on $E$ (Def. \ref{defpairing}) &  symmetric, $\langle e,e'\rangle=\langle e',e\rangle$, no sign \\
\hline
Lie$^*$ super-antisymmetry (Def. \ref{defLiestarsuper}) & $\mu(f\cdot a\boxtimes b)=-(-1)^{|a||b|}\sigma_{1,2}\mu(f^{\mathrm{sw}}\cdot b\boxtimes a)$ \\
\hline
Chiral Clifford grading (Lem. \ref{lemCliffordcompat}) & $k$ factors of $L$ $\Rightarrow$ parity $k\bmod2$ \\
\hline
Parity of $\HH^0(X,E)$ in $\OBV$ (Def. \ref{defOBV}) & parity of $E$ (odd, for us) \\
\hline
Parity of $\HH^1(X,E)$ in $\OBV$ (Def. \ref{defOBV}) & \emph{opposite} to parity of $E$ (even, for us) - Lemma \ref{lemgradingforce} \\
\hline
$\DBV$ (Def. \ref{defOBV}) & odd operator; $\DBV^2=0$ (Prop. \ref{propBVaxioms}) \\
\hline
BV bracket sign (eq. \eqref{eqBVbracket}) & $\{a,b\}=\DBV(ab)-(\DBV a)b-(-1)^{|a|}a\DBV b$ \\
\hline
Pfaffian contraction (eq. \eqref{eqpfaffiandef}) & $\mathrm{sgn}(\sigma)$ for $\sigma\in S_{2m}$, vs.\ permanent (no sign) in the bosonic case of \cite{Gui23} \\
\hline
Super-cyclicity of the trace (Lem. \ref{lemcyclicity}) & $\trom\mathbf p\Wch(ab)=(-1)^{|a||b|}\trom\mathbf p\Wch(ba)$ \\
\hline
\end{tabular}
\end{center}

%%%%%%%%%%%%%%%%%%%%%%%%%%%%%%%%%%%%%%%%%%%%%%%%%%%%%%%%%%%%%%%%%%%%%%%%%%%%%%%%
\section*{Acknowledgements}

The author is supported by the Institute of Mathematics, Academy of Sciences
of the Czech Republic (RVO 67985840). 

\medskip
\noindent\textbf{Data Availability.}
Data sharing is not applicable to this article as no datasets were generated
or analysed during the current study.

\medskip
\noindent\textbf{Declarations}

\medskip
\noindent\textbf{Conflict of interest.}
The author has no conflicts of interest to declare that are relevant to the
content of this article.

%%%%%%%%%%%%%%%%%%%%%%%%%%%%%%%%%%%%%%%%%%%%%%%%%%%%%%%%%%%%%%%%%

\end{document}